\documentclass{article}[12pt]

\usepackage{makeidx}

\usepackage[english]{babel}
\usepackage[latin1]{inputenc}
\usepackage[T1]{fontenc}
\usepackage{amsmath, amsthm}
\usepackage{amscd}
\usepackage{amssymb}
\usepackage{amssymb}
\usepackage{latexsym}
\usepackage{amsmath}
\usepackage{amsfonts}
\usepackage{amsthm}

\usepackage{url}

\usepackage{color}

\usepackage{enumerate}
\usepackage{graphicx}
\usepackage{float}

\usepackage{lipsum}
\usepackage{tikz}
\usetikzlibrary{matrix,arrows,decorations.pathmorphing}
\usepackage[english]{babel}
\usepackage[latin1]{inputenc}
\usepackage[T1]{fontenc}
\usepackage{amsmath}
\usepackage{amssymb}
\usepackage{amsthm}
\usepackage{enumerate}

\newtheorem{thm}{Theorem}[section]
\newtheorem{theorem}[thm]{Theorem}

\newtheorem{corollary}[thm]{Corollary}

\newtheorem{proposition}[thm]{Proposition}

\theoremstyle{definition}
\newtheorem{example}[thm]{Example}

\makeatletter
\newcommand{\subalign}[1]{%
  \vcenter{%
    \Let@ \restore@math@cr \default@tag
    \baselineskip\fontdimen10 \scriptfont\tw@
    \advance\baselineskip\fontdimen12 \scriptfont\tw@
    \lineskip\thr@@\fontdimen8 \scriptfont\thr@@
    \lineskiplimit\lineskip
    \ialign{\hfil$\m@th\scriptstyle##$&$\m@th\scriptstyle{}##$\crcr
      #1\crcr
    }%
  }
}
\makeatother

\begin{document}

\centerline{\Large \bf On the action of the Koszul map}

\centerline{\Large \bf on the enveloping algebra}

\centerline{\Large \bf  of the general linear
Lie algebra}

\bigskip

\centerline{A. Brini and A. Teolis}
\centerline{\it $^\flat$ Dipartimento di Matematica, Universit\`{a} di
Bologna }
 \centerline{\it Piazza di Porta S. Donato, 5. 40126 Bologna. Italy.}
\centerline{\footnotesize e-mail of corresponding author: andrea.brini@unibo.it}
\medskip

\begin{abstract}

We describe a linear \emph{equivariant isomorphism} $\mathcal{K}$ from the enveloping algebra $\mathbf{U}(gl(n))$
to the  algebra ${\mathbb C}[M_{n,n}] \cong \mathbf{Sym}(gl(n))$ of polynomials in the  entries of a ``generic''
square matrix of order $n$.

The isomorphism $\mathcal{K}$ maps any {\textit{Capelli bitableau}} $[S|T]$ in $\mathbf{U}(gl(n))$ to the
{\textit{(determinantal) bitableau}} $(S|T)$ in ${\mathbb C}[M_{n,n}]$
and any {\textit{Capelli *-bitableau}} $[S|T]^*$ in $\mathbf{U}(gl(n))$ to the
{\textit{(permanental) *-bitableau}} $(S|T)^*$ in ${\mathbb C}[M_{n,n}]$.

These results are  far-reaching generalizations of the pioneering result of J.-L. Koszul \cite{Koszul-BR}
on the Capelli determinant in $\mathbf{U}(gl(n))$ (see, e.g. \cite{Procesi-BR}, \cite{Weyl-BR}).

We introduce {\textit{column}} Capelli bitableaux and  *-bitableaux
in Section \ref{column expansion}; since they are mapped by the isomorphism $\mathcal{K}$
to {\textit{monomials}} in ${\mathbb C}[M_{n,n}]$, this isomorphism can be regarded as a  sharpened
version of the PBW isomorphism for the enveloping algebra $\mathbf{U}(gl(n))$.

Since the center $\boldsymbol{\zeta}(n)$ of $\mathbf{U}(gl(n))$ equals the subalgebra of invariants
$\mathbf{U}(gl(n))^{Ad_{gl(n)}}$,
then
$$
\mathcal{K} \big[ \boldsymbol{\zeta}(n) \big] = {\mathbb C}[M_{n,n}]^{ad_{gl(n)}}.
$$

\end{abstract}


central elements; Capelli determinants.


\section{Introduction}

The starting points of the present work are (\cite{Brini3-BR}, \cite{Brini4-BR}):
\begin{itemize}

\item [--] The linear operator $\mathcal{B} : {\mathbb C}[M_{n,n}] \rightarrow \mathbf{U}(gl(n))$
that maps any {\textit{(determinantal) bitableau}} $(S|T)$ in ${\mathbb C}[M_{n,n}]$ to the
{\textit{Capelli bitableau}} $[S|T]$ in $\mathbf{U}(gl(n))$.

\item [--] The linear operator $\mathcal{B^*} : {\mathbb C}[M_{n,n}] \rightarrow \mathbf{U}(gl(n))$
that maps any {\textit{(permanental) bitableau}} $(S|T)^*$ in ${\mathbb C}[M_{n,n}]$ to the
{\textit{Capelli *-bitableau}} $[S|T]^*$ in $\mathbf{U}(gl(n))$.
\end{itemize}

The map
$$\mathcal{K} : \mathbf{U}(gl(n)) \rightarrow {\mathbb C}[M_{n,n}] \cong \mathbf{Sym}(gl(n))$$
introduced by Koszul in 1981 \cite{Koszul-BR} is proved to be the inverse of both $\mathcal{B}$ and
$\mathcal{B}^*$. Then  $\mathcal{B}$,  $\mathcal{B}^*$,
$\mathcal{K}$ are vector space isomorphisms and $\mathcal{B} = \mathcal{B}^*$.

Since the set of {\textit{standard}}  bitableaux is a basis
of ${\mathbb C}[M_{n,n}]$ (\cite{drs-BR}, \cite{DKR-BR}, \cite{DEP-BR}, \cite{rota-BR}),
then the set of {\textit{standard}} Capelli bitableaux
is a  basis of $\mathbf{U}(gl(n))$.
Since the set of {\textit{costandard}}  *-bitableaux is a basis
of ${\mathbb C}[M_{n,n}]$,
then the set of {\textit{costandard}} Capelli *-bitableaux
is a  basis of $\mathbf{U}(gl(n))$.

Some  of these topics were treated in a sketchy way in the present author's notes
\cite{Brini3-BR}, \cite{Brini4-BR} (in the more general setting
of superalgebras), in a rather cumbersome notation and almost without proofs.
The main novelty of the present approach is the major role played by  
{\it{column Capelli bitableaux}} and  {\it{column Capelli *-bitableaux}};
although they are far from being ``monomials'' in the
enveloping algebra $\mathbf{U}(gl(n))$, their images with respect to the Koszul isomorphism 
$\mathcal{K}$ are indeed monomials in the polynomial algebra ${\mathbb C}[M_{n,n}]$.
Therefore, column Capelli bitableaux and column Capelli *-bitableaux play the same  
role in $\mathbf{U}(gl(n))$ that monomials play in ${\mathbb C}[M_{n,n}]$ and this 
leads to a new and transparent presentation.

The  expressions of column Capelli bitableaux and column Capelli *-bitableaux as elements of $\mathbf{U}(gl(n))$
can be simply computed.

Capelli bitableaux and Capelli *-bitableaux
expand  - up to  a global sign - into column Capelli bitableaux
just in the same way as \textit{determinantal}
bitableaux and \textit{permanental} *-bitableaux
expand into the  corresponding monomials
in  ${\mathbb C}[M_{n,n}]$ (Laplace expansions).

The isomorphism $\mathcal{B} = \mathcal{K}^{-1}$ maps any \emph{right symmetrized bitableau} 
$(S|\fbox{$T$}) \in {\mathbb C}[M_{n,n}]$ (\cite{Brini1-BR}, \cite{Brini2-BR})
to the {\textit{right Young-Capelli bitableau}} $[S|\fbox{$T$}]$ in $\mathbf{U}(gl(n))$. The basis
of {\textit{standard}} right Young-Capelli bitableaux acts in a  remarkable way on the
{\textit{Gordan-Capelli basis}} of {\textit{standard}} right symmetrized bitableaux.
Moreover, the elements of the \emph{Schur-Sahi-Okounkov} basis of the center $\boldsymbol{\zeta}(n)$ of
$\mathbf{U}(gl(n))$ (\emph{quantum immanants} \cite{Sahi1-BR}, \cite{Okounkov-BR}, \cite{Okounkov1-BR}, 
\cite{OkOlsh-BR}) admit  quite effective presentations as linear combinations of right Young-Capelli 
bitableaux as well as of Capelli immanants \cite{Brini5-BR} and \cite{BriniTeolis-BR}.

The Koszul map $\mathcal{K}$ is proved to be an \emph{equivariant} isomorphism with respect to the adjoint
representations of $gl(n)$ on $\mathbf{U}(gl(n))$ and ${\mathbb C}[M_{n,n}]$ (\emph{polarization action}), 
respectively. 
Since the center $\boldsymbol{\zeta}(n)$ of $\mathbf{U}(gl(n))$ is the subalgebra of 
$Ad_{gl(n)}$-invariants of $\mathbf{U}(gl(n))^{Ad_{gl(n)}}$,
then
$$
\mathcal{K} \big[ \boldsymbol{\zeta}(n) \big] = {\mathbb C}[M_{n,n}]^{ad_{gl(n)}},
$$
where ${\mathbb C}[M_{n,n}]^{ad_{gl(n)}}$  
is the subalgebra of 
$ad_{gl(n)}$-invariants of ${\mathbb C}[M_{n,n}]$.

\section{Determinantal {\it{Young bitableaux}}, permanental {\it{Young *-bitableaux}} 
and {\it{right symmetrized bitableaux}} in the polynomial algebra ${\mathbb C}[M_{n,n}]$}\label{sec 2}

Let
$$
{\mathbb C}[M_{n,n}] =    {\mathbb C}[(i|j)]_{i,j=1,\ldots,n}
$$
be the polynomial algebra in the (commutative)
 ``generic" entries $(i|j)$ of the matrix:
$$
M_{n,n} = \left[ (i|j) \right]_{i, j=1,\ldots,n}=
 \left(
 \begin{array}{ccc}
 (1|1) & \ldots & (1|n) \\
 \vdots  &        & \vdots \\
 (n|1) & \ldots & (n|n) \\
 \end{array}
 \right).
$$

Given the \emph{standard basis} $\big\{ e_{ij}; \ i, j = 1, 2, \ldots, n \big\}$
of the \emph{general linear Lie algebra} $gl(n)$, the map
$
e_{ij} \mapsto (i|j)
$
induces an isomorphism $\mathbf{Sym}(gl(n)) \cong {\mathbb C}[M_{n,n}]$.

Let $\omega = i_1i_2 \cdots i_p$, $\varpi = j_1j_1 \cdots j_p$ be  words on the alphabet $ \{1, 2, \ldots, n \}$.

Following \cite{rota-BR} and \cite{Brini1-BR},
the {\it{biproduct}} of the two words $\omega$ and $\varpi$
\begin{equation}\label{biproduct}
(\omega|\varpi) = (i_1i_2 \cdots i_p|j_1j_2 \cdots j_p)
\end{equation}
is the  \emph{signed minor}:
$$
(\omega|\varpi) = (-1)^{p \choose 2} \ 
det \Big( \ (i_r|j_s) \ \Big)_{r, s = 1, 2, \ldots, p} \in {\mathbb C}[M_{n,n}]. 
$$

Let $S = (\omega_1, \omega_2, \ldots, \omega_p)$ and $T = (\varpi_1, \varpi_2, \ldots, \varpi_p)$ be Young tableaux on
$\{1, 2, \ldots, n \}$  of the same shape $\lambda$.

Following again  \cite{rota-BR} and \cite{Brini1-BR},  the (determinantal) {\it{Young bitableau}}
\begin{equation}
(S|T) =
\left(
\begin{array}{c}
\omega_1\\ \omega_2\\ \vdots\\ \omega_p
\end{array}
\right| \left.
\begin{array}{c}
\varpi_1\\ \varpi_2\\ \vdots\\  \varpi_p
\end{array}
\right)
\end{equation}\label{product}
is the \emph{signed} product of the biproducts of the pairs of corresponding  rows:
\begin{equation}\label{bitableau}
(S|T) =
\textbf{s} \ (\omega_1|\varpi_1)(\omega_2|\varpi_2) \cdots (\omega_p|\varpi_p),
\end{equation}
where
\begin{equation}\label{crossing sign}
\textbf{s}  = (-1)^{\ell(\omega_2)\ell(\varpi_1)+\ell(\omega_3)(\ell(\varpi_1)+\ell(\varpi_2))
+ \cdots +\ell(\omega_p)(\ell(\varpi_1)+\ell(\varpi_2)+\cdots+\ell(\varpi_{p-1}))},
\end{equation}
and the symbol $\ell(w)$ denotes the length of the word $w$.

The {\it{*-biproduct}} of the two words $\omega$ and $\varpi$
\begin{equation}
(\omega|\varpi)^* = (i_1i_2 \cdots i_p|j_1j_2 \cdots j_p)^*
\end{equation}
is the \emph{permanent}:
$$
(\omega|\varpi)^* =  \ per \Big( \ (i_r|j_s) \ \Big)_{r, s = 1, 2, \ldots, p} \in {\mathbb C}[M_{n,n}].
$$

Let $S = (\omega_1, \omega_2, \ldots, \omega_p)$ and $T = (\varpi_1, \varpi_2, \ldots, \varpi_p)$ be Young tableaux on
$\{1, 2, \ldots, n \}$  of the same shape $\lambda$.

Following again  \cite{rota-BR} and \cite{Brini1-BR},  the (permanental) {\it{Young *-bitableau}}
\begin{equation}
(S|T)^* =
\left(
\begin{array}{c}
\omega_1\\ \omega_2\\ \vdots\\ \omega_p
\end{array}
\right| \left.
\begin{array}{c}
\varpi_1\\ \varpi_2\\ \vdots\\  \varpi_p
\end{array}
\right)^*
\end{equation}
is the product of the *-biproducts of the pairs of corresponding  rows:
\begin{equation}\label{*-bitableau}
(S|T)^* =
 (\omega_1|\varpi_1)^*(\omega_2|\varpi_2)^* \cdots (\omega_p|\varpi_p)^*.
\end{equation}

A \emph{column} Young tableau of \emph{depth} $h$ is a tableau of  shape $(1^h)$.
Then for a column Young bitableau, we have:
\begin{equation}\label{column monomial}
\left(
\begin{array}{c}
i_1\\  i_2 \\ \vdots \\ i_h
\end{array}
\right| \left.
\begin{array}{c}
j_1\\  j_2 \\ \vdots \\ j_h
\end{array}
\right)
=
(-1)^{h \choose 2}(i_1|j_1)(i_2|j_2) \cdots (i_h|j_h)
\end{equation}
and for a column Young *-bitableau, we have:
\begin{equation}\label{column *-monomial}
\left(
\begin{array}{c}
i_1\\  i_2 \\ \vdots \\ i_h
\end{array}
\right| \left.
\begin{array}{c}
j_1\\  j_2 \\ \vdots \\ j_h
\end{array}
\right)^*
=
(i_1|j_1)(i_2|j_2) \cdots (i_h|j_h).
\end{equation}

We recall the definition of the \emph{right symmetrized bitableau} $(S|\fbox{$T$}))$ 
(see, e.g. \cite{Brini1-BR}):
\begin{equation}\label{symmetrized bitableau}
(S|\fbox{$T$})) = \sum_{\overline{T}} \ (S|\overline{T}), 
\end{equation}
where the sum is extended over {\textit{all}} $\overline{T}$ column permuted of $T$ (hence, repeated entries in
a column give rise to multiplicities).

\begin{example}
$$
\left(
\begin{array}{cc}
 1 & 3 \\  2 & 4
\end{array}
\right| \left.
\fbox{$
\begin{array}{cc}
1 & 2 \\ 1 & 3
\end{array}
$} \
\right)
=
2  \left(
\begin{array}{cc}
 1 & 3 \\  2 & 4
\end{array}
\right| \left.
\begin{array}{cc}
1 & 2 \\ 1 & 3
\end{array}
\right)
+
2  \left(
\begin{array}{cc}
 1 & 3 \\  2 & 4
\end{array}
\right| \left.
\begin{array}{cc}
1 & 3 \\ 1 & 2
\end{array}
\right).
$$
\end{example}

We recall same elementary definitions and notational conventions. 
Given a partition (shape) 
$\lambda = (\lambda_1 \geq \lambda_2 \geq \cdots \geq \lambda_p) \vdash n$,
let $\widetilde{\lambda} = 
(\widetilde{\lambda}_1, \widetilde{\lambda}_2 \geq \cdots \geq \widetilde{\lambda}_q) \vdash n$
denote its \emph{conjugate} partition, where 
$\widetilde{\lambda}_s = \# \{t; \lambda_t \geq s \}$. Similarly, given a Young tableau 
$S$ of shape $sh(S) = \lambda$, let $\widetilde{S}$ denote its conjugate (dual) Young tableau.
In plain words, $\widetilde{S}$ is the tableau whose rows are the columns of $S$ and whose 
shape is $sh(\widetilde{S}) = \widetilde{\lambda}$. 
A Young tableau $X$ on the (linearly ordered) set
$L = \{ 1 < 2 <  \cdots < n \}$ is said to be \emph{standard} whenever its rows are increasing 
from left to right and its columns are non decreasing downwards. 
In a dual way, a Young tableau $Y$ is said to be \emph{costandard} whenever its conjugate
Young tableau $\widetilde{Y}$ is standard.

We recall the basis theorems for \emph{standard determinantal bitableaux} (see, e.g. \cite{drs-BR},
\cite{DKR-BR}, \cite{DEP-BR}), 
\emph{costandard permanental *-bitableaux} 
\cite{rota-BR} and  \emph{right symmetrized bitableaux} \cite{Brini1-BR}, respectively.

\begin{proposition}\label{bases thms}
The sets
\begin{itemize}

\item [--] $\Big \{ (S|T); \ sh(S) =sh(T) =  \lambda, \ \lambda_1 \leq n, \ S,  T \ standard \Big \}$,

\item [--] $\Big \{ (U|V)^*; \ sh(U) =sh(V) =  \mu, \ \widetilde{\mu_1} \leq n, \ U,  V \ costandard \Big \}$,

\item [--] $\Big\{ (S|\fbox{$T$});  \ sh(S) = sh(T) = \lambda, \ \lambda_1 \leq n, \ S, T \ standard \Big\}$
\end{itemize}

are linear bases of ${\mathbb C}[M_{n,n}]$. 
\end{proposition}

\section{Polarization operators and Lie algebra representations of $gl(n)$ 
on ${\mathbb C}[M_{n,n}]$ and $\mathbf{U}(gl(n))$}

Given $i, j = 1, 2, \ldots, n$, the \emph{left polarization operator}
(of $j$ to $i$)
$D^{\textit{l}}_{ij}$ is the linear operator from ${\mathbb C}[M_{n,n}]$ 
to itself defined by the conditions:
\begin{itemize}

\item [--] $D^{\textit{l}}_{ij}$ is a derivation

\item [--] $D^{\textit{l}}_{ij} \big( (h|k) \big) = \delta_{jh} (i|k)$  for  every $k.$

\end{itemize}

Similarly, the \emph{right polarization operator}
(of $i$ to $j$)
$D^{\textit{r}}_{ji}$ is the linear operator from ${\mathbb C}[M_{n,n}]$ 
to itself defined by the conditions:
\begin{itemize}

\item [--] $D^{\textit{r}}_{ji}$ is a derivation

\item [--] $D^{\textit{r}}_{ji} \big( (h|k) \big) = \delta_{ik} (h|j)$  for  every $h$.

\end{itemize}

In the following, we consider \emph{three} Lie algebra representations
$$
gl(n) \rightarrow End_{{\mathbb C}} \big[ {\mathbb C}[M_{n,n}] \big]
$$
and the corresponding Lie modules. 

\begin{enumerate}

\item
The \emph{left} (covariant) representation  $\rho^{\textit{l}}$ is defined by setting
$$
\rho^{\textit{l}} : e_{ij} \mapsto D^{\textit{l}}_{ij}.
$$

\item
The \emph{right} (contravariant) representation  $\rho^{\textit{r}}$ is defined  by setting
$$
\rho^{\textit{r}} : e_{ij} \mapsto -D^{\textit{r}}_{ji}.
$$

\item  Notice that $D^{\textit{l}}_{ij}D^{\textit{r}}_{hk} = D^{\textit{r}}_{hk}D^{\textit{l}}_{ij}$.
The  \emph{adjoint} representation $ad_{gl(n)}$ is defined by setting
$$
ad_{gl(n)} : e_{ij} \mapsto D^{\textit{l}}_{ij} - D^{\textit{r}}_{ji}.
$$
\end{enumerate}

Given $i, j = 1, 2, \ldots, n$, consider   the linear operator $T_{ij}$
from $\mathbf{U}(gl(n))$ to itself defined by 
setting
$$
T_{ij} \big(\textbf{M} \big) = e_{ij} \textbf{M} - \textbf{M}  e_{ij},
$$
for every $\textbf{M} \in \mathbf{U}(gl(n))$.

We recall that $T_{ij}$ is the unique derivation of $\mathbf{U}(gl(n))$ such that 
$$
T_{ij} \big(e_{st} \big) = \delta_{js} e_{it} - \delta_{it} e_{sj} = \big[e_{ij}, e_{st} \big] 
$$
for every $s, t = 1, 2, \ldots, n$.
Hence 
$$
T_{ij} \circ T_{hk} - T_{hk} \circ T_{ij} = \delta_{jh} T_{ik} - \delta_{ik} T_{hj}.
$$
The Lie algebra representation
$$
Ad_{gl(n)} : gl(n) \rightarrow End_{{\mathbb C}} \big[ \mathbf{U}(gl(n)) \big]
$$
$$
e_{ij} \mapsto T_{ij}
$$
is the \emph{adjoint} representation of $\mathbf{U}(gl(n))$ on itself.

\section{The superalgebraic approach to the enveloping algebra $\mathbf{U}(gl(n))$}\label{sec 3}

In this Section, we provide a synthetic presentation   of the {\it{superalgebraic method of virtual variables}} for $gl(n)$.

This method was developed  by the present authors
for the   general linear Lie superalgebras $gl(m|n)$ \cite{KAC1-BR}, in the series of notes
\cite{Bri-BR}, \cite{BriUMI-BR}, \cite{Brini1-BR}, \cite{Brini2-BR}, \cite{Brini3-BR},
\cite{Brini4-BR}.

The technique of virtual variables is  an extension of Capelli's method of  {\textit{ variabili ausilarie}}
(Capelli \cite{Cap4-BR}, see also Weyl \cite{Weyl-BR}).

Capelli introduced the technique of {\textit{ variabili ausilarie}} in order to manage symmetrizer operators in
terms of polarization operators and to  simplify the study of some skew-symmetrizer operators (namely, the famous central Capelli operator).

Capelli's idea was well suited to
treat symmetrization, but it did not work in the same efficient way while dealing with skew-symmetrization.

One had to wait the introduction of the notion of {\textit{superalgebras}} 
(see,e.g. \cite{Scheu-BR},  \cite{KAC1-BR})
 to have the right conceptual framework to treat symmetry and skew-symmetry in one and the same way.
To the best of our knowledge, the first mathematician who intuited the connection between Capelli's idea and superalgebras was
Koszul in $1981$  \cite{Koszul-BR}.
In particular, Koszul proved that the classical determinantal Capelli operator can be rewritten - in a much simpler way - by adding to the symbols to be dealt with an extra auxiliary symbol that obeys to different commutation relations.

The superalgebraic method of virtual variables allows us to express remarkable classes 
of elements in $\mathbf{U}(gl(n))$
as images - with respect to the {\textit{Capelli devirtualization epimorphism}} - of simple
{\it{monomials}} and to obtain transparent combinatorial 
descriptions of their actions on irreducible $gl(n)-$modules.

This method is very well suited for the study of the \emph{polarization action}
of $\mathbf{U}(gl(n))$ on ${\mathbb C}[M_{n,n}]$ and for the study of 
the \emph{center} of $\mathbf{U}(gl(n))$.

\subsection{The superalgebras ${\mathbb C}[M_{m_0|m_1+n,n}]$ and $gl(m_0|m_1+n)$ }

\subsubsection{The general linear Lie super algebra $gl(m_0|m_1+n)$}
 Given a vector space $V_n$ of dimension $n$, we will regard it as a subspace of a $
\mathbb{Z}_2-$graded vector space
 $W = W_0 \oplus W_1$, where
$$
W_0 = V_{m_0}, \qquad W_1 = V_{m_1} \oplus V_n.
$$
The  vector spaces
$V_{m_0}$ and $V_{m_1}$ (we assume that 
$dim(V_{m_0})=m_0$ and $dim(V_{m_1})=m_1$ are ``sufficiently large'') are called
the  {\textit{positive virtual (auxiliary)
vector space}},  the {\textit{negative virtual (auxiliary) vector space}}, respectively, and $V_n$ 
is called the {\textit{(negative) proper vector space}}.

 The inclusion $V_n \subset W$ induces a natural embedding of the ordinary general 
linear Lie algebra $gl(n)$ of $V_n$ into the
 {\textit{auxiliary}}
general linear Lie {\it{superalgebra}} $gl(m_0|m_1+n)$ of $W = W_0 \oplus W_1$ (see, e.g. \cite{KAC1-BR}, 
\cite{Scheu-BR}).

Let
$
A_0 = \{ \alpha_1, \ldots, \alpha_{m_0} \},$  $A_1 = \{ \beta_1, \ldots, \beta_{m_1} \},$
$L = \{ 1, 2,  \ldots, n \}$
denote \emph{fixed  bases} of $V_{m_0}$, $V_{m_1}$ and $V_n$, respectively; 
therefore $|\alpha_s| = 0 \in \mathbb{Z}_2,$
and $|\beta_t| = |i|   = 1 \in \mathbb{Z}_2.$

Let
$$
\{ e_{a, b}; a, b \in A_0 \cup A_1 \cup L \}, \qquad |e_{a, b}| =
|a|+|b| \in \mathbb{Z}_2
$$
be the standard $\mathbb{Z}_2-$homogeneous basis of the Lie superalgebra $gl(m_0|m_1+n)$ provided by the
elementary matrices. The elements $e_{a, b} \in gl(m_0|m_1+n)$ are $\mathbb{Z}_2-$homogeneous of
$\mathbb{Z}_2-$degree $|e_{a, b}| = |a| + |b|.$

The superbracket of the Lie superalgebra $gl(m_0|m_1+n)$ has the following explicit form:
$$
\left[ e_{a, b}, e_{c, d} \right] = \delta_{bc} \ e_{a, d} - (-1)^{(|a|+|b|)(|c|+|d|)} \delta_{ad}  \ e_{c, b},
$$
$a, b, c, d \in A_0 \cup A_1 \cup L.$

In the following, the elements of the sets $A_0, A_1, L$ will be called
\emph{positive virtual symbols}, \emph{negative virtual symbols} and \emph{negative proper symbols},
respectively.

\subsubsection{The  supersymmetric algebra ${\mathbb C}[M_{m_0|m_1+n,n}]$}

We regard the commutative algebra ${\mathbb C}[M_{n,n}]$
as a subalgebra of the \textit{``auxiliary'' supersymmetric algebra}
$$
{\mathbb C}[M_{m_0|m_1+n,n}] = {\mathbb C}\big[ (\alpha_s|j), (\beta_t|j), (i|j) \big]
$$
 generated by the ($\mathbb{Z}_2$-graded) variables $(\alpha_s|j), (\beta_t|j), (i|j)$,
$j = 1, 2, \ldots, n$,
 where
 $$
 |(\alpha_s|j)| = |\alpha_s| = 1 \in \mathbb{Z}_2, \quad \   |(\beta_t|j)| = |\beta_t|+1 =  0 \in \mathbb{Z}_2
 $$
and $|(i|j)| = |i|+|j| = 0,$ subject to the commutation relations:
$$
(a|h)(b|k) = (-1)^{|(a|h)||(b|k)|} \ (b|k)(a|h),
$$
for 
$a, b \in  \{ \alpha_1, \ldots, \alpha_{m_0} \} \cup \{ \beta_1, \ldots, \beta_{m_1} \} \cup \{1, 2, \ldots , n\}.$

In plain words, all the variables commute each other, with the exception of
pairs of variables $(\alpha_s|j), (\alpha_t|j)$ that skew-commute:
$$
(\alpha_s|j) (\alpha_t|j) = - (\alpha_t|j) (\alpha_s|j).
$$

In the standard notation of multilinear algebra, we have:
\begin{align*}
{\mathbb C}[M_{m_0|m_1+n,n}]
& \cong \Lambda \big[ W_0 \otimes P_n \big]
\otimes      {\mathrm{Sym}} \big[ W_1  \otimes P_n \big] \\
 & =
 \Lambda \big[ V_{m_0} \otimes P_n \big]
\otimes      {\mathrm{Sym}} \big[ (V_{m_1} \oplus V_n)  \otimes P_n \big]
\end{align*}
where $P_n = (P_n)_1$ denotes the trivially  $\mathbb{Z}_2-$graded  vector space with distinguished basis 
$\{j; \ |j|=1, \ j = 1, 2, \ldots, n \}.$

The algebra  ${\mathbb C}[M_{m_0|m_1+n,n}]$ is a supersymmetric 
\textit{$\mathbb{Z}_2-$graded algebra} (superalgebra), whose
$\mathbb{Z}_2-$graduation is  inherited by the natural one in the exterior algebra.

\subsubsection{Left superderivations and left superpolarizations}

A {\it{left superderivation}} $D^{\textit{l}}$ ($\mathbb{Z}_2-$homogeneous of degree $|D^{\textit{l}}|$) (see, e.g. \cite{Scheu-BR}, \cite{KAC1-BR}) on
${\mathbb C}[M_{m_0|m_1+n,n}]$ is an element of the superalgebra $End_\mathbb{C}[\mathbb{C}[M_{m_0|m_1+n,n}]]$
that satisfies "Leibniz rule"
$$
D^{\textit{l}}(\textbf{p} \cdot \textbf{q}) = D^{\textit{l}}(\textbf{p}) \cdot \textbf{q} + 
(-1)^{|D^{\textit{l}}||\textbf{p}|} \textbf{p} \cdot D^{\textit{l}}(\textbf{q}),
$$
for every $\mathbb{Z}_2-$homogeneous of degree $|\textbf{p}|$ element $\textbf{p} \in \mathbb{C}[M_{m_0|m_1+n,n}].$

Given two symbols $a, b \in A_0 \cup A_1 \cup L$, the {\textit{left superpolarization}} $D^{\textit{l}}_{a,b}$ 
of $b$ to $a$
is the unique {\it{left}} superderivation of ${\mathbb C}[M_{m_0|m_1+n,n}]$ of $\mathbb{Z}_2-$degree 
$|D^{\textit{l}}_{a,b}| = |a| + |b| \in \mathbb{Z}_2$ such that
$$
D^{\textit{l}}_{a,b} \left( (c|j) \right) = \delta_{bc} \ (a|j), \ c \in A_0 \cup A_1 \cup L, \ j = 1, \ldots, n.
$$

Informally, we say that the operator $D^{\textit{l}}_{a,b}$ {\it{annihilates}} the symbol $b$ and {\it{creates}} the symbol $a$.

\subsubsection{The superalgebra ${\mathbb C}[M_{m_0|m_1+n,n}]$ as a $\mathbf{U}(gl(m_0|m_1+n))$-module}

Since
$$
D^{\textit{l}}_{a,b}D^{\textit{l}}_{c,d} -(-1)^{(|a|+|b|)(|c|+|d|)}D^{\textit{l}}_{c,d}D^{\textit{l}}_{a,b} =
\delta_{b,c}D^{\textit{l}}_{a,d} -(-1)^{(|a|+|b|)(|c|+|d|)}\delta_{a,d}D^{\textit{l}}_{c,b},
$$
the map
$$
e_{a,b} \mapsto D^{\textit{l}}_{a,b}, \qquad a, b \in A_0 \cup A_1 \cup L
$$
is a Lie superalgebra morphism from $gl(m_0|m_1+n)$ to $End_\mathbb{C}\big[\mathbb{C}[M_{m_0|m_1+n,n}]\big]$
and, hence, it uniquely defines a
representation:
$$
\varrho : \mathbf{U}(gl(m_0|m_1+n)) \rightarrow End_\mathbb{C}[\mathbb{C}[M_{m_0|m_1+n,n}]].
$$

In the following, we always regard the superalgebra $\mathbb{C}[M_{m_0|m_1+n,n}]$ as a $\mathbf{U}(gl(m_0|m_1+n))-$supermodule,
with respect to the action induced by the representation $\varrho$:
$$
e_{a,b} \cdot \mathbf{p} = D^{\textit{l}}_{a,b}(\mathbf{p}),
$$
for every $\mathbf{p} \in {\mathbb C}[M_{m_0|m_1+n,n}].$

We recall that  $\mathbf{U}(gl(m_0|m_1+n))-$module  $\mathbb{C}[M_{m_0|m_1+n,n}]$
is  a semisimple module, whose simple submodules are - up to isomorphism - {\it{Schur supermodules}} 
(see, e.g. \cite{Brini1-BR}, \cite{Brini2-BR}, \cite{Bri-BR}. For a more traditional presentation, see also 
\cite{CW-BR}).

Clearly, $\mathbf{U}(gl(0|n)) = \mathbf{U}(gl(n))$ is a subalgebra of $\mathbf{U}(gl(m_0|m_1+n))$
and the subalgebra $\mathbb{C}[M_{n,n}]$ is a $\mathbf{U}(gl(n))-$submodule of  $\mathbb{C}[M_{m_0|m_1+n,n}]$.

\subsection{The virtual algebra $Virt(m_0+m_1,n)$ and the virtual
presentations of elements in $\mathbf{U}(gl(n))$}

We say that a product
$$
e_{a_mb_m} \cdots e_{a_1b_1} \in \mathbf{U}(gl(m_0|m_1+n)), 
\quad a_i, b_i \in A_0 \cup A_1 \cup L, \ i= 1, \ldots, m
$$
is an {\textit{irregular expression}} whenever
  there exists a right subword
$$e_{a_i,b_i} \cdots e_{a_2,b_2} e_{a_1,b_1},$$
$i \leq m$ and a
virtual symbol $\gamma \in A_0 \cup A_1$ such that
\begin{equation}\label{irrexpr-BR}
 \# \{j;  b_j = \gamma, j \leq i \}  >  \# \{j;  a_j = \gamma, j < i \}.
\end{equation}

The meaning of an irregular expression in terms of the action of  $\mathbf{U}(gl(m_0|m_1+n))$  
by left superpolarization on
the algebra $\mathbb{C}[M_{m_0|m_1+n,n}]$ is that there exists a
virtual symbol $\gamma$ and a right subsequence in which the symbol $\gamma$ is \emph{annihilated} 
more times than it was already \emph{created} and, therefore, the action of an irregular expression
on the algebra $\mathbb{C}[M_{n,n}]$ is \emph{zero}. 

\begin{example}
Let $\gamma \in  A_0 \cup A_1$ and $x_i, x_j \in L.$ The product
$$
e_{\gamma,x_j} e_{x_i,\gamma} e_{x_j,\gamma} e_{\gamma,x_i}
$$
is an irregular expression.
\end{example}\qed

Let $\mathbf{Irr}$   be
the {\textit{left ideal}} of $\mathbf{U}(gl(m_0|m_1+n))$ generated by the set of
irregular expressions.

\begin{proposition}
The superpolarization action
of any element of $\mathbf{Irr}$ on the subalgebra $\mathbb C[M_{n,n}] \subset \mathbb{C}[M_{m_0|m_1+n,n}]$ - via the representation $\varrho$ -
is identically zero.
\end{proposition}

\begin{proposition}\emph{(\cite{Brini3-BR}, \cite{BriUMI-BR})}
The sum ${\mathbf{U}}(gl(0|n)) + \mathbf{Irr}$ is a direct sum of vector subspaces of $\mathbf{U}(gl(m_0|m_1+n)).$
\end{proposition}

\begin{proposition}\emph{(\cite{Brini3-BR}, \cite{BriUMI-BR})}
The direct sum vector subspace $\mathbf{U}(gl(0|n)) \oplus \mathbf{Irr}$ is a \emph{subalgebra} 
of $\mathbf{U}(gl(m_0|m_1+n)).$
\end{proposition}

The subalgebra
$$
Virt(m_0+m_1,n) = \mathbf{U}(gl(0|n)) \oplus \mathbf{Irr} \subset {\mathbf{U}}(gl(m_0|m_1+n)).
$$
is called the {\textit{virtual algebra}}.  

The proof of the following proposition is immediate from the definitions.
\begin{proposition}
The left ideal  $\mathbf{Irr}$ of ${\mathbf{U}}(gl(m_0|m_1+n))$
is a two sided ideal of $Virt(m_0+m_1,n).$
\end{proposition}

The {\textit{Capelli devirtualization epimorphism}} is the surjection
$$
\mathfrak{p} : Virt(m_0+m_1,n) = \mathbf{U}(gl(0|n)) \oplus \mathbf{Irr} \twoheadrightarrow \mathbf{U}(gl(0|n)) = \mathbf{U}(gl(n))
$$
with $Ker(\mathfrak{p}) = \mathbf{Irr}.$

Any element in $\textbf{M} \in Virt(m_0+m_1,n)$ defines an element in
$\textbf{m} \in \mathbf{U}(gl(n))$ - via the map $\mathfrak{p}$ -
 and $\textbf{M}$ is called a \textit{virtual
presentation} of $\textbf{m}$.

Since the map $\mathfrak{p}$  a surjection, any element
$\mathbf{m} \in \mathbf{U}(gl(n))$ admits several virtual
presentations. In the sequel, we even take virtual presentations
as the \emph{true definition} of special elements in $\mathbf{U}(gl(n)),$ 
and this method will turn out to be quite effective.

Recall that ${\mathbf{U}}(gl(m_0|m_1+n))$ is a Lie module with respect 
to the adjiont representation $Ad_{gl(m_0|m_1+n)}$. Since $gl(n) = gl(0|n)$ 
is a Lie subalgebra of $gl(m_0|m_1+n)$, ${\mathbf{U}}(gl(m_0|m_1+n))$ is a $gl(n)-$module 
with respect to the adjoint action $Ad_{gl(n)}$ of $gl(n)$.

The following results follow from the definitions.

\begin{proposition} The virtual algebra $Virt(m_0+m_1,n)$ is a submodule 
of ${\mathbf{U}}(gl(m_0|m_1+n))$ with respect to the adjoint action $Ad_{gl(n)}$ of $gl(n)$.
\end{proposition}

\begin{proposition}\label{rappresentazione aggiunta-BR} The Capelli  epimorphism 
$$
\mathfrak{p} : Virt(m_0+m_1,n)  \twoheadrightarrow \mathbf{U}(gl(n))
$$ is an $Ad_{gl(n)}-$\emph{equivariant} map.
\end{proposition}

\begin{corollary} The isomorphism $\mathfrak{p}$ maps
any  $Ad_{gl(n)}-$invariant element $\mathbf{m} \in Virt(m_0+m_1,n)$  to a \emph{central} 
element of $\mathbf{U}(gl(n))$.
\end{corollary}

\textit{Balanced monomials} are  elements of the algebra ${\mathbf{U}}(gl(m_0|m_1+n))$
 of the form:
\begin{itemize}\label{defbalanced monomials-BR}
\item [--] $e_{{i_1},\gamma_{p_1}} \cdots e_{{i_k},\gamma_{p_k}} \cdot
e_{\gamma_{p_1},{j_1}} \cdots e_{\gamma_{p_k},{j_k}},$
\item [--]
$e_{{i_1},\theta_{q_1}} \cdots e_{{i_k},\theta_{q_k}} \cdot
e_{\theta_{q_1},\gamma_{p_1}} \cdots e_{\theta_{q_k},\gamma_{p_k}} \cdot
e_{\gamma_{p_1},{j_1}} \cdots e_{\gamma_{p_k},{j_k}},$
\item [--] and so on,
\end{itemize}
where
$i_1, \ldots, i_k, j_1, \ldots, j_k \in L,$
i.e., the $i_1, \ldots, i_k, j_1, \ldots, j_k$ are $k$
proper (negative) symbols, and the
$\gamma_{p_1}, \ldots, \gamma_{p_k}, \ldots, \theta_{q_1}, \ldots, \theta_{q_k}, \ldots$ are
virtual symbols.
In plain words, a balanced monomial is product of two or more factors  where the
rightmost one  \textit{annihilates} (by superpolarization)
the $k$ proper symbols $ j_1, \ldots, j_k$ and
\textit{creates} (by superpolarization) some virtual symbols;
 the leftmost one  \textit{annihilates} all the virtual symbols
and \textit{creates} the $k$ proper symbols $ i_1, \ldots, i_k$;
between these two factors, there might be further factors that annihilate
 and create  virtual symbols only.

\begin{proposition}\emph{(\cite{Brini1-BR}, \cite{Brini2-BR}, \cite{Bri-BR}, \cite{BriUMI-BR})}
Every balanced monomial belongs to $Virt(m_0+m_1,n)$. Hence,
the Capelli epimorphism $\mathfrak{p}$ maps  balanced monomials to elements of $\mathbf{U}(gl(n)).$
\end{proposition}

Let $S$ and $T$ be the Young tableaux
\begin{equation}\label{tableaux}
S = \left(
\begin{array}{llllllllllllll}
i_{p_1}  \ldots    \ldots     \ldots     i_{p_{\lambda_1}} \\
i_{q_1}   \ldots  \ldots               i_{q_{\lambda_2}} \\
 \ldots  \ldots   \\
i_{r_1} \ldots i_{r_{\lambda_m}}
\end{array}
\right),
\quad
T = \left(
\begin{array}{llllllllllllll}
j_{s_1}  \ldots    \ldots     \ldots     j_{s_{\lambda_1}} \\
j_{t_1}   \ldots  \ldots               j_{t_{\lambda_2}} \\
 \ldots  \ldots   \\
j_{v_1} \ldots j_{v_{\lambda_m}}
\end{array}
\right).
\end{equation}

To the pair $(S,T)$, we associate the {\it{bitableau monomial}}:
\begin{equation*}\label{BitMon}
e_{S,T} =
e_{i_{p_1}, j_{s_1}}\cdots e_{i_{p_{\lambda_1}}, j_{s_{\lambda_1}}}
e_{i_{q_1}, j_{t_1}}\cdots e_{i_{q_{\lambda_2}}, j_{t_{\lambda_2}}}
 \cdots  \cdots
e_{i_{r_1}, j_{v_1}}\cdots e_{i_{r{\lambda_p}}, j_{v_{\lambda_p}}}
\end{equation*}
in ${\mathbf{U}}(gl(m_0|m_1+n)).$

Let  $\beta_1, \ldots, \beta_{\lambda_1} \in A_1$, $\alpha_1, \ldots, \alpha_p \in A_0$
be sets of  negative and positive  \emph{virtual symbols}, respectively.
Set
\begin{equation*}\label{Deruyts and Coderuyts}
D_{\lambda} = \left(
\begin{array}{llllllllllllll}
\beta_1  \ldots    \ldots     \ldots     \beta_{\lambda_1}     \\
\beta_1   \ldots  \ldots               \beta_{\lambda_2} \\
 \ldots  \ldots   \\
\beta_1 \ldots \beta_{\lambda_p}
\end{array}
\right), \qquad
C_{\lambda} = \left(
\begin{array}{llllllllllllll}
\alpha_1  \ldots    \ldots     \ldots     \alpha_1    \\
\alpha_2   \ldots  \ldots               \alpha_2 \\
 \ldots  \ldots   \\
\alpha_p \ldots \alpha_p
\end{array}
\right)
\end{equation*}.

The tableaux $D_{\lambda}$ and $C_{\lambda}$ are called the {\it{virtual Deruyts and Coderuyts tableaux}}
of shape $\lambda,$ respectively.

Given a pair of Young tableaux  $S, T$ of the same shape $\lambda$ on the proper alphabet $L$,
consider the elements
\begin{equation}\label{determinantal}
e_{S,C_{\lambda}} \ e_{C_{\lambda},T} \in {\mathbf{U}}(gl(m_0|m_1+n)),
\end{equation}
\begin{equation}\label{permanental}
e_{S,\widetilde{D_{\widetilde{\lambda}}}} \ e_{\widetilde{D_{\widetilde{\lambda}}},T} \in {\mathbf{U}}(gl(m_0|m_1+n)),
\end{equation}
\begin{equation}\label{rightYoungCapelli}
e_{S,C_{\lambda}} \ e_{C_{\lambda},D_{\lambda}}  \   e_{D_{\lambda},T} \in {\mathbf{U}}(gl(m_0|m_1+n)).
\end{equation}

Since elements (\ref{determinantal}), (\ref{permanental}) and (\ref{rightYoungCapelli}) are balanced monomials in
${\mathbf{U}}(gl(m_0|m_1+n))$,  they belong to the subalgebra $Virt(m_0+m_1,n)$.

We set
$$
\mathfrak{p} \Big( e_{S,C_{\lambda}} \ e_{C_{\lambda},T}    \Big) =  [S|T]  \in {\mathbf{U}}(gl(n)),
$$
and call the element $[S|T]$ a {\textit{Capelli bitableau}} \cite{Brini3-BR}, \cite{Brini4-BR}.

We set
$$
\mathfrak{p} \Big( e_{S,\widetilde{D_{\widetilde{\lambda}}}} \ e_{\widetilde{D_{\widetilde{\lambda}}},T}    \Big) =  [S|T]^*  \in {\mathbf{U}}(gl(n)),
$$
and call the element $[S|T]^*$ a {\textit{Capelli *-bitableau}} \cite{Brini3-BR}, \cite{Brini4-BR}.

We set
$$
\mathfrak{p} \Big( e_{S,C_{\lambda}} \ e_{C_{\lambda},D_{\lambda}}  \   e_{D_{\lambda},T} \Big)
=  [S| \fbox{$T$}]           \in {\mathbf{U}}(gl(n)).
$$
and call the element $[S| \fbox{$T$}] $ a {\textit{right Young-Capelli bitableau}} \cite{Brini2-BR}.

\section{The \textit{bitableaux correspondence} maps $\mathcal{B}$ and $\mathcal{B}^*$
and the Koszul map $\mathcal{K}$}\label{bbk maps}

\begin{theorem}\label{operator B}
The \emph{bitableaux correspondence} map
$$
\mathcal{B} : (S|T) \mapsto [S|T]
$$
uniquely extends to a linear map
$$
\mathcal{B} : {\mathbb C}[M_{n,n}] \cong \mathbf{Sym}(gl(n)) \rightarrow \mathbf{U}(gl(n)).
$$
\end{theorem}
\begin{proof} We recall that bitableaux  and Capelli bitableaux
satisfy the \textbf{same} (determinantal)  \emph{straightening laws} 
in ${\mathbb C}[M_{n,n}]$ and  $\mathbf{U}(gl(n))$, respectively 
(\cite{Brini3-BR}, Proposition $7$). The  straightening laws
imply that standard  (determinantal) bitableaux span ${\mathbb C}[M_{n,n}]$
(see, e.g. \cite{drs-BR}, \cite{DEP-BR}, \cite{DKR-BR}); furhermore, 
standard bitableaux are linearly independent. Then, the map $\mathcal{B}$ is a uniquely
defined linear operator.
\end{proof}

\begin{theorem}\label{operator B^*}
The \emph{*-bitableaux correspondence} map
$$
\mathcal{B}^* : (S|T)^* \mapsto [S|T]^*
$$
uniquely extends to a linear map
$$
\mathcal{B}^* : {\mathbb C}[M_{n,n}]  \cong \mathbf{Sym}(gl(n)) \rightarrow \mathbf{U}(gl(n)).
$$
\end{theorem}
\begin{proof} The proof is essentially the same as the proof of Theorem \ref{operator B},
just by replacing the determinantal straightening laws with the permanental straightening laws, 
and standard  (determinantal) bitableaux with costandard  (permanental) bitableaux.
Notice that both arguments are special cases of the superalgebraic version of the straightening laws 
and of the standard basis theorem (\cite{rota-BR}, \cite{Bri-BR}).
\end{proof}

Given $i, j = 1, 2, \ldots, n$, let
$$
\rho_{ij} : {\mathbb C}[M_{n,n}] \rightarrow {\mathbb C}[M_{n,n}]
$$
be the linear operator
$$
\rho_{ij}(\mathbf{p}) = D^{\textit{l}}_{ij} (\mathbf{p}) + (i|j) \cdot \mathbf{p},
\quad for \ every \ \mathbf{p} \in {\mathbb C}[M_{n,n}].
$$

\begin{proposition}\label{main identity} We have:
$$
[\rho_{ij},\rho_{hk}] = \rho_{ij}  \rho_{hk} - \rho_{hk}  \rho_{ij} =
\delta_{jh} \rho_{ik} - \delta_{ik} \rho_{hj}.
$$
\end{proposition}\qed

By the universal property of $\mathbf{U}(gl(n))$, Proposition \ref{main identity}
implies

\begin{proposition} The map
$$
e_{ij} \mapsto\rho_{ij}, \quad e_{ij} \in gl(n)
$$
defines an associative algebra morphism
$$
\tau : \mathbf{U}(gl(n)) \rightarrow End_{\mathbb C}[{\mathbb C}[M_{n,n}]].
$$
\end{proposition}\qed

Let $\varepsilon_1$ be the linear map \emph{evaluation at} $1$
$$
\varepsilon_1 : End_{\mathbb C}[{\mathbb C}[M_{n,n}]]  \rightarrow {\mathbb C}[M_{n,n}],
$$
$$
\varepsilon_1(\rho) = \rho(1) \in {\mathbb C}[M_{n,n}], \quad for \ every \ \rho \in End_{\mathbb C}[{\mathbb C}[M_{n,n}]].
$$

The \emph{Koszul map} \cite{Koszul-BR} is the (\emph{linear}) composition map
$$
\mathcal{K} : \mathbf{U}(gl(n)) \rightarrow {\mathbb C}[M_{n,n}] \cong \mathbf{Sym}(gl(n))  ,
$$
$$
\mathcal{K} = \varepsilon_1 \circ \tau.
$$

\begin{proposition} We have:
\begin{enumerate}

\item
$
\mathcal{K}(e_{i_1j_1}e_{i_2j_2} \cdots e_{i_hj_h}) = \rho_{i_1j_1}\rho_{i_2j_2} \cdots \rho_{i_hj_h}(1),
$
\ $e_{i_pj_p} \in gl(n)$, $p =1 2, \ldots ,h$.

\item
$
\mathcal{K}(e_{ij}\mathbf{P}) = \rho_{ij}(\mathcal{K}(\mathbf{P})),
$
\ for every $\mathbf{P} \in \mathbf{U}(gl(n))$, $\ e_{ij} \in gl(n)$.
\end{enumerate}
\end{proposition} \qed

\section{Expansion formulae for
\textit{column Capelli bitableaux} and
\textit{column Capelli *-bitableaux}}\label{column expansion}

Consider the \textit{column Capelli bitableau}
$$
\left[
\begin{array}{c}
i_1\\  i_2 \\ \vdots \\ i_h
\end{array}
\right| \left.
\begin{array}{c}
j_1\\  j_2 \\ \vdots \\ j_h
\end{array}
\right]
=
\mathfrak{p} \Big(e_{i_1 \alpha_1} \cdots e_{i_h \alpha_h}e_{\alpha_1 j_1} \cdots e_{\alpha_h j_h} \Big) 
\in \mathbf{U}(gl(n)),
$$
(where $\alpha_1, \ldots, \alpha_h$ are arbitrary \emph{distict positive virtual} symbols)
and the \textit{column Capelli *-bitableau}
$$
\left[
\begin{array}{c}
i_1\\  i_2 \\ \vdots \\ i_h
\end{array}
\right| \left.
\begin{array}{c}
j_1\\  j_2 \\ \vdots \\ j_h
\end{array}
\right]^*
=
\mathfrak{p} \Big(e_{i_1 \beta_1} \cdots e_{i_h \beta_h}e_{\beta_1 j_1} \cdots e_{\beta_h j_h} \Big) \in \mathbf{U}(gl(n))
$$
(where $\beta_1, \ldots, \beta_h$ are arbitrary \emph{distict negative virtual} symbols).

Remember that the \emph{proper symbols} $i_1, \ldots,i_h, \ j_1, \ldots,  j_h \in 
L = \{1, \ 2, \ \ldots, \ n \}$
are assumed to be \emph{negative}.

From the definitions, it follows
\begin{equation}\label{sign Capelli column}
\left[
\begin{array}{c}
i_1\\  i_2 \\ \vdots \\ i_h
\end{array}
\right| \left.
\begin{array}{c}
j_1\\  j_2 \\ \vdots \\ j_h
\end{array}
\right]
= (-1)^{n \choose 2}
\left[
\begin{array}{c}
i_1\\  i_2 \\ \vdots \\ i_h
\end{array}
\right| \left.
\begin{array}{c}
j_1\\  j_2 \\ \vdots \\ j_h
\end{array}
\right]^*.
\end{equation}

From the definitions, we infer
\begin{proposition}
Column Capelli bitableaux and column Capelli *-bitableaux are
row-commutative as elements of $\mathbf{U}(gl(n))$:
\begin{enumerate}

\item
$$
\left[
\begin{array}{c}
i_1\\  i_2 \\ \vdots \\ i_h
\end{array}
\right| \left.
\begin{array}{c}
j_1\\  j_2 \\ \vdots \\ j_h
\end{array}
\right]
=
\left[
\begin{array}{c}
i_{\sigma(1)}\\  i_{\sigma(2)} \\ \vdots \\ i_{\sigma(h)}
\end{array}
\right| \left.
\begin{array}{c}
j_{\sigma(1)}\\  j_{\sigma(2)} \\ \vdots \\ j_{\sigma(h)}
\end{array}
\right], \quad \sigma \in \textbf{S}_h,
$$

\item
$$
\left[
\begin{array}{c}
i_1\\  i_2 \\ \vdots \\ i_h
\end{array}
\right| \left.
\begin{array}{c}
j_1\\  j_2 \\ \vdots \\ j_h
\end{array}
\right]^*
=
\left[
\begin{array}{c}
i_{\sigma(1)}\\  i_{\sigma(2)} \\ \vdots \\ i_{\sigma(h)}
\end{array}
\right| \left.
\begin{array}{c}
j_{\sigma(1)}\\  j_{\sigma(2)} \\ \vdots \\ j_{\sigma(h)}
\end{array}
\right]^*, \quad \sigma \in \textbf{S}_h,
$$

\end{enumerate} \qed

\end{proposition}

We provide two basic expansion formulae that describe the effect of picking out (on the left hand side)
the  first row of column Capelli bitableaux and column Capelli *-bitableaux.
These formulae play a crucial role in the theory of the Koszul map $\mathcal{K}$,
and  provide a simple way to compute the \emph{actual} forms of
column Capelli bitableaux and column Capelli *-bitableaux as elements of $\mathbf{U}(gl(n))$.

\begin{proposition}\label{column exp} We have:

\

\begin{enumerate}
\item
\begin{align*}
& \left[
\begin{array}{c}
i_1\\  i_2 \\ \vdots \\ i_h
\end{array}
\right| \left.
\begin{array}{c}
j_1\\  j_2 \\ \vdots \\ j_h
\end{array}
\right]  =
\\
=&
\ (-1)^{h-1} e_{i_1j_1} \left[
\begin{array}{c}
  i_2 \\ \vdots \\ i_h
\end{array}
\right| \left.
\begin{array}{c}
  j_2 \\ \vdots \\ j_h
\end{array}
\right]   +
(-1)^{h-2} \sum_{k=2}^h \ \delta_{i_kj_1} \ \left[
\begin{array}{c}
 i_2 \\ \vdots \\ i_1 \\ \vdots \\ i_h
\end{array}
\right| \left.
\begin{array}{c}
 j_2 \\ \vdots \\ j_k \\ \vdots \\ j_h
\end{array}
\right] \in \mathbf{U}(gl(n)).
\end{align*}

\item
\begin{align*}
& \left[
\begin{array}{c}
i_1\\  i_2 \\ \vdots \\ i_h
\end{array}
\right| \left.
\begin{array}{c}
j_1\\  j_2 \\ \vdots \\ j_h
\end{array}
\right]^*  =
\\
=&
\  e_{i_1j_1} \left[
\begin{array}{c}
  i_2 \\ \vdots \\ i_h
\end{array}
\right| \left.
\begin{array}{c}
  j_2 \\ \vdots \\ j_h
\end{array}
\right]^*   -
\sum_{k=2}^h \ \delta_{i_kj_1} \ \left[
\begin{array}{c}
 i_2 \\ \vdots \\ i_1 \\ \vdots \\ i_h
\end{array}
\right| \left.
\begin{array}{c}
 j_2 \\ \vdots \\ j_k \\ \vdots \\ j_h
\end{array}
\right]^* \in \mathbf{U}(gl(n)).
\end{align*}

\end{enumerate}
\end{proposition}

\begin{proof}

By definition,
$$
\left[
\begin{array}{c}
i_1\\  i_2 \\ \vdots \\ i_{h-1} \\ i_h
\end{array}
\right| \left.
\begin{array}{c}
j_1 \\ j_2 \\ \vdots\\ j_{h-1} \\ j_h
\end{array}
\right] =
$$
\begin{align*}
=& \ \mathfrak{p} \big[ e_{{i_1},\alpha_1} e_{{i_2},\alpha_2}   \cdots e_{{i_{h-1}},\alpha_{h-1}} e_{{i_{h}},\alpha_h} \cdot
e_{\alpha_1, {j_1}} e_{\alpha_2, {j_2}}   \cdots e_{\alpha_{h-1}, {j_{h-1}}} e_{\alpha_h, {j_{h}}} \big] =                                            \\
=& \ \mathfrak{p} \big[
  - e_{{i_1},\alpha_1} e_{{i_2},\alpha_2}   \cdots e_{{i_{h-1}},\alpha_{h-1}} e_{\alpha_1, {j_1}}
 e_{{i_{h}},\alpha_h} \cdot e_{\alpha_2, {j_2}}   \cdots e_{\alpha_{h-1}, {j_{h-1}}} e_{\alpha_h, {j_{h}}}
\\
&\phantom{\ \mathfrak{p} \big[} + e_{{i_1},\alpha_1} e_{{i_2},\alpha_2}   \cdots e_{{i_{h-1}},\alpha_{h-1}}
  \cdot \delta_{i_h, j_1} e_{\alpha_1, \alpha_h} e_{\alpha_2, {j_2}}   \cdots  e_{\alpha_{h-1}, {j_{h-1}}} e_{\alpha_h, {j_{h}}} \big] =
\\
=& \ \mathfrak{p} \big[
  - e_{{i_1},\alpha_1} e_{{i_2},\alpha_2}   \cdots e_{{i_{h-1}},\alpha_{h-1}} e_{\alpha_1, {j_1}}
 e_{{i_{h}},\alpha_h} \cdot e_{\alpha_2, {j_2}}   \cdots e_{\alpha_{h-1}, {j_{h-1}}} e_{\alpha_h, {j_{h}}}
\\
&\phantom{\ \mathfrak{p} \big[} +
e_{{i_1},\alpha_1} e_{{i_2},\alpha_2}   \cdots e_{{i_{h-1}},\alpha_{h-1}}
  \cdot \delta_{i_h, j_1} e_{\alpha_2, {j_2}}   \cdots  e_{\alpha_{h-1}, {j_{h-1}}} e_{\alpha_1, {j_{h}}} \big].
\end{align*}
Notice that
$$
 \delta_{i_h, j_1} \
e_{{i_1},\alpha_1} e_{{i_2},\alpha_2}   \cdots e_{{i_{h-1}},\alpha_{h-1}}
  \cdot  e_{\alpha_2, {j_2}}   \cdots  e_{\alpha_{h-1}, {j_{h-1}}} e_{\alpha_1, {j_{h}}}  =
$$
$$
 \delta_{i_h, j_1} \ (-1)^{h - 2} \
e_{{i_1},\alpha_1} e_{{i_2},\alpha_2}   \cdots e_{{i_{h-1}},\alpha_{h-1}} \cdot
 e_{\alpha_1, {j_{h}}}    e_{\alpha_2, {j_2}}   \cdots  e_{\alpha_{h-1}, {j_{h-1}}}
$$
as elements of the algebra ${\mathbf{U}}(gl(m_0|m_1+n)).$

Therefore, the summand
$$
\mathfrak{p} \big[ e_{{i_1},\alpha_1} e_{{i_2},\alpha_2}   \cdots e_{{i_{h-1}},\alpha_{h-1}}
  \cdot \delta_{i_h, j_1} e_{\alpha_2, {j_2}}   \cdots  e_{\alpha_{h-1}, {j_{h-1}}} e_{\alpha_1, {j_{h}}} \big]
$$
equals
$$(-1)^{h - 2} \
\delta_{i_h, j_1} \
\left[
\begin{array}{c}
i_1\\  i_2 \\ \vdots \\ i_{h-1}
\end{array}
\right| \left.
\begin{array}{c}
j_h \\ j_2 \\ \vdots\\ j_{h-1}
\end{array}
\right].
$$
By repeating the above procedure of moving
left the element $e_{\alpha_1, {j_1}}$ - using the commutator identities in ${\mathbf{U}}(gl(m_0|m_1+n))$ -
we finally get
$$
\left[
\begin{array}{c}
i_1\\  i_2 \\ \vdots \\ i_{h-1} \\ i_h
\end{array}
\right| \left.
\begin{array}{c}
j_1 \\ j_2 \\ \vdots\\ j_{h-1} \\ j_h
\end{array}
\right] =
$$
\begin{align*}
=& \ \mathfrak{p} \big[ (-1)^{h - 1}
  e_{{i_1},\alpha_1} e_{\alpha_1, {j_1}}e_{{i_2},\alpha_2}   \cdots e_{{i_{h-1}},\alpha_{h-1}}
 e_{{i_{h}},\alpha_h} \cdot e_{\alpha_2, {j_2}}   \cdots e_{\alpha_{h-1}, {j_{h-1}}} e_{\alpha_h, {j_{h}}}
\\
&\phantom{\ \mathfrak{p} \big[} +
 \sum_{i = 0}^{h - 2} \ (-1)^i \
e_{{i_1},\alpha_1}    \cdots \delta_{i_{h-i}, j_1}\widehat{e_{{i_{h-i}},\alpha_{h-i}}}e_{\alpha_1,\alpha_{h-i}} \cdots
e_{{i_{h}},\alpha_{h}}
  \cdot  e_{\alpha_2, {j_2}}   \cdots   e_{\alpha_h, {j_{h}}} \big]
\\
=&  \ \mathfrak{p} \big[ (-1)^{h - 1}
  e_{{i_1},\alpha_1} e_{\alpha_1, {j_1}}e_{{i_2},\alpha_2}   \cdots e_{{i_{h-1}},\alpha_{h-1}}
 e_{{i_{h}},\alpha_h} \cdot e_{\alpha_2, {j_2}}   \cdots e_{\alpha_{h-1}, {j_{h-1}}} e_{\alpha_h, {j_{h}}}
\\
&\phantom{\ \mathfrak{p} \big[} +
\sum_{i = 0}^{h - 2} \ (-1)^i \
e_{{i_1},\alpha_1} \cdot\cdot \delta_{i_{h-i}, j_1} \cdots
e_{{i_{h}},\alpha_{h}}
  \cdot  e_{\alpha_2, {j_2}} \cdot\cdot  e_{\alpha_1,j_{h-i}}\cdot \cdot e_{\alpha_h, {j_{h}}} \big].
\end{align*}

Notice that the summand
$$
(-1)^i \ \delta_{i_{h-i}, j_1} \
e_{{i_1},\alpha_1} \cdot\cdot \delta_{i_{h-i}, j_1} \cdots
e_{{i_{h}},\alpha_{h}}
  \cdot  e_{\alpha_2, {j_2}} \cdot \cdot  e_{\alpha_1,j_{h-i}}\cdot \cdot e_{\alpha_h, {j_{h}}}
$$
equals
$$
(-1)^i \ \delta_{i_{h-i}, j_1} \ (-1)^{h-i-2} \times
$$
$$
e_{{i_1},\alpha_1}    \cdots \widehat{e_{i_{h-i},\alpha_{h-i}}} \cdots
e_{{i_{h}},\alpha_{h}}
  \cdot e_{\alpha_1,j_{h-i}}  e_{\alpha_2, {j_2}}   \cdots \widehat{e_{\alpha_{h-i}, {j_{h-i}}}} \dots  e_{\alpha_h, {j_{h}}}
$$
as elements of the algebra ${\mathbf{U}}(gl(m_0|m_1+n)).$

Hence
$$
\mathfrak{p} \left[ (-1)^i \ \delta_{i_{h-i}, j_1} \
e_{{i_1},\alpha_1}    \cdots \widehat{e_{{i_{h-i}},\alpha_{h-i}}}e_{\alpha_1,\alpha_{h-i}} \cdots
e_{{i_{h}},\alpha_{h}}
  \cdot  e_{\alpha_2, {j_2}}   \cdots   e_{\alpha_h, {j_{h}}} \right]
$$
equals
$$ (-1)^{h-2} \ \delta_{i_{h-i}, j_1} \
\left[
\begin{array}{c}
i_1\\  i_2 \\ \vdots \\  i_{h-i-1} \\\widehat{i_{h-i}} \\ i_{h-i+1} \\ i_h
\end{array}
\right| \left.
\begin{array}{c}
j_{h-i} \\ j_2 \\ \vdots \\  j_{h-i-1} \\ \widehat{j_{h-i}} \\ j_{h-i+1} \\ j_h
\end{array}
\right].
$$
Furthermore
$$
\mathfrak{p} \big[ (-1)^{h - 1}
  e_{{i_1},\alpha_1} e_{\alpha_1, {j_1}}e_{{i_2},\alpha_2}   \cdots e_{{i_{h-1}},\alpha_{h-1}}
 e_{{i_{h}},\alpha_h} \cdot e_{\alpha_2, {j_2}}   \cdots e_{\alpha_{h-1}, {j_{h-1}}} e_{\alpha_h, {j_{h}}}
\big] =
$$
$$
= (-1)^{h - 1} \ e_{i_1, j_1} \
\left[
\begin{array}{c}
i_2 \\ \vdots \\ i_{h-1} \\ i_h
\end{array}
\right| \left.
\begin{array}{c}
j_2 \\ \vdots\\ j_{h-1} \\ j_h
\end{array}
\right].
$$

Since column Capelli bitableaux are row-commutative, by setting $k = h -i$ 
we proved the first expansion identity.
The second expansion identity can be proved in a similar way.
\end{proof}

\begin{example}
\begin{align*}
\left[
\begin{array}{c}
 1 \\ 2 \\ 3 \\ 2
\end{array}
\right| \left.
\begin{array}{c}
  2 \\ 3 \\ 4 \\ 3
\end{array}
\right] &=
- e_{12}\left[
\begin{array}{c}
  2 \\ 3 \\ 2
\end{array}
\right| \left.
\begin{array}{c}
  3 \\ 4 \\ 3
\end{array}
\right]
+
\left[
\begin{array}{c}
  1 \\ 3 \\ 2
\end{array}
\right| \left.
\begin{array}{c}
  3 \\ 4 \\ 3
\end{array}
\right]
+
\left[
\begin{array}{c}
  2 \\ 3 \\ 1
\end{array}
\right| \left.
\begin{array}{c}
  3 \\ 4 \\ 3
\end{array}
\right]
\\
&=
- e_{12}\left[
\begin{array}{c}
  2 \\ 3 \\ 2
\end{array}
\right| \left.
\begin{array}{c}
  3 \\ 4 \\ 3
\end{array}
\right]
+
\ 2 \ \left[
\begin{array}{c}
  1 \\ 3 \\ 2
\end{array}
\right| \left.
\begin{array}{c}
  3 \\ 4 \\ 3
\end{array}
\right]
\\
&=
- e_{12} \Big( e_{23} \left[
\begin{array}{c}
3 \\ 2
\end{array}
\right| \left.
\begin{array}{c}
4 \\ 3
\end{array}
\right]
-
\left[
\begin{array}{c}
2 \\ 2
\end{array}
\right| \left.
\begin{array}{c}
4 \\ 3
\end{array}
\right] \Big)
\\
& \phantom{=} \
+
2 \Big( e_{13}
\left[
\begin{array}{c}
3 \\ 2
\end{array}
\right| \left.
\begin{array}{c}
4 \\ 3
\end{array}
\right]
 -
\left[
\begin{array}{c}
1 \\ 2
\end{array}
\right| \left.
\begin{array}{c}
4 \\ 3
\end{array}
\right] \Big)
\\
&=
e_{12}e_{23}e_{34}e_{23} -
e_{12}e_{24}e_{23}
- 2 e_{13}e_{34}e_{23}
+ 2 e_{14}e_{23} \in \mathbf{U}(gl(4)).
\end{align*}
\end{example}\qed

Notice that, for $h = 1$, $[i|j] = [i|j]^* = e_{ij}$. Then,
from Proposition \ref{column exp}, it follows 
\begin{corollary}\label{generators}
The family of column Capelli bitableaux  (*-bitableaux) is a system
of linear generators of $\mathbf{U}(gl(n))$.
\end{corollary}

\section{Main results}

\begin{proposition}\label{K column}
\begin{align*}
\mathcal{K} \big(  \left[
\begin{array}{c}
i_1\\  i_2 \\ \vdots \\ i_h
\end{array}
\right| \left.
\begin{array}{c}
j_1\\  j_2 \\ \vdots \\ j_h
\end{array}
\right]  \big) =&
\left(
\begin{array}{c}
i_1\\  i_2 \\ \vdots \\ i_h
\end{array}
\right| \left.
\begin{array}{c}
j_1\\  j_2 \\ \vdots \\ j_h
\end{array}
\right)
\\
=& \ (-1)^{h \choose 2} (i_1|j_1)(i_2|j_2) \dots (i_h|j_h)
 \in {\mathbb C}[M_{n,n}] \cong \mathbf{Sym}(gl(n)).
\end{align*}
\end{proposition}

\begin{proof}
\begin{align*}
& \mathcal{K} \big(  \left[
\begin{array}{c}
i_1\\  i_2 \\ \vdots \\ i_h
\end{array}
\right| \left.
\begin{array}{c}
j_1\\  j_2 \\ \vdots \\ j_h
\end{array}
\right]  \big) =
\\
=&
\ (-1)^{h-1} \mathcal{K} \big( e_{i_1j_1} \left[
\begin{array}{c}
  i_2 \\ \vdots \\ i_h
\end{array}
\right| \left.
\begin{array}{c}
  j_2 \\ \vdots \\ j_h
\end{array}
\right]  \big) +
(-1)^{h-2} \mathcal{K} \big( \sum_{k=2}^h \ \delta_{i_kj_1} \ \left[
\begin{array}{c}
 i_2 \\ \vdots \\ i_1 \\ \vdots \\ i_h
\end{array}
\right| \left.
\begin{array}{c}
 j_2 \\ \vdots \\ j_k \\ \vdots \\ j_h
\end{array}
\right]  \big)
\\
=&
\ (-1)^{h-1} \rho_{i_1j_1} \big(\mathcal{K} \big(  \left[
\begin{array}{c}
  i_2 \\ \vdots \\ i_h
\end{array}
\right| \left.
\begin{array}{c}
  j_2 \\ \vdots \\ j_h
\end{array}
\right]  \big) \big) +
(-1)^{h-2} \mathcal{K} \big( \sum_{k=2}^h \ \delta_{i_kj_1} \ \left[
\begin{array}{c}
 i_2 \\ \vdots \\ i_1 \\ \vdots \\ i_h
\end{array}
\right| \left.
\begin{array}{c}
 j_2 \\ \vdots \\ j_k \\ \vdots \\ j_h
\end{array}
\right]  \big)
\\
=&
\ (-1)^{h-1} D^{\textit{l}}_{i_1j_1} \big( \left(
\begin{array}{c}
 i_2 \\ \vdots \\ i_h
\end{array}
\right| \left.
\begin{array}{c}
  j_2 \\ \vdots \\ j_h
\end{array}
\right) \big) + (-1)^{h-1} (i_1|j_1)\left(
\begin{array}{c}
 i_2 \\ \vdots \\ i_h
\end{array}
\right| \left.
\begin{array}{c}
  j_2 \\ \vdots \\ j_h
\end{array}
\right)
\\
&
 \qquad \qquad + (-1)^{h-2} \sum_{k=2}^h \ \delta_{i_kj_1} \ \left(
\begin{array}{c}
 i_2 \\ \vdots \\ i_1 \\ \vdots \\ i_h
\end{array}
\right| \left.
\begin{array}{c}
 j_2 \\ \vdots \\ j_k \\ \vdots \\ j_h
\end{array}
\right)
\\
=&
\ (-1)^{h-1} (i_1|j_1)\left(
\begin{array}{c}
 i_2 \\ \vdots \\ i_h
\end{array}
\right| \left.
\begin{array}{c}
  j_2 \\ \vdots \\ j_h
\end{array}
\right) =
\left(
\begin{array}{c}
i_1\\  i_2 \\ \vdots \\ i_h
\end{array}
\right| \left.
\begin{array}{c}
j_1\\  j_2 \\ \vdots \\ j_h
\end{array}
\right).
\end{align*}
\end{proof}

\begin{example} Consider the column Capelli bitableau
$$
\left[
\begin{array}{c}
1\\  2  \\ 3
\end{array}
\right| \left.
\begin{array}{c}
2 \\ 1  \\ 1
\end{array}
\right] = e_{12} \left[
\begin{array}{c}
2  \\ 3
\end{array}
\right| \left.
\begin{array}{c}
 1  \\ 1
\end{array}
\right] - \left[
\begin{array}{c}
1  \\ 3
\end{array}
\right| \left.
\begin{array}{c}
 1  \\ 1
\end{array}
\right] = - e_{12}e_{21}e_{31} + e_{11}e_{31} \in \mathbf{U}(gl(n)).
$$

We have
\begin{align*}
\mathcal{K} \big( \left[
\begin{array}{c}
1\\  2  \\ 3
\end{array}
\right| \left.
\begin{array}{c}
2 \\ 1  \\ 1
\end{array}
\right] \big) &= \mathcal{K} \big( - e_{12}e_{21}e_{31} + e_{11}e_{31} \big)
\\
&=
\ \left(
\begin{array}{c}
1\\  2  \\ 3
\end{array}
\right| \left.
\begin{array}{c}
2 \\ 1  \\ 1
\end{array}
\right)
\\
&= -(1|2)(2|1)(3|1) \in {\mathbb C}[M_{n,n}] \cong \mathbf{Sym}(gl(n)).
\end{align*}
\end{example}\qed

\begin{proposition}\label{K *-column}
\begin{align*}
\mathcal{K} \big(  \left[
\begin{array}{c}
i_1\\  i_2 \\ \vdots \\ i_h
\end{array}
\right| \left.
\begin{array}{c}
j_1\\  j_2 \\ \vdots \\ j_h
\end{array}
\right]^*  \big) =&
\left(
\begin{array}{c}
i_1\\  i_2 \\ \vdots \\ i_h
\end{array}
\right| \left.
\begin{array}{c}
j_1\\  j_2 \\ \vdots \\ j_h
\end{array}
\right)^*
\\
=& \ (i_1|j_1)(i_2|j_2) \dots (i_h|j_h)
 \in {\mathbb C}[M_{n,n}] \cong \mathbf{Sym}(gl(n)).
\end{align*}
\end{proposition}
\begin{proof}
\begin{align*}
 \mathcal{K} \big(  \left[
\begin{array}{c}
i_1\\  i_2 \\ \vdots \\ i_h
\end{array}
\right| \left.
\begin{array}{c}
j_1\\  j_2 \\ \vdots \\ j_h
\end{array}
\right]^*  \big) =&
\\
=&
\ \mathcal{K} \big( e_{i_1j_1} \left[
\begin{array}{c}
  i_2 \\ \vdots \\ i_h
\end{array}
\right| \left.
\begin{array}{c}
  j_2 \\ \vdots \\ j_h
\end{array}
\right]^*  \big) -
 \mathcal{K} \big( \sum_{k=2}^h \ \delta_{i_kj_1} \ \left[
\begin{array}{c}
 i_2 \\ \vdots \\ i_1 \\ \vdots \\ i_h
\end{array}
\right| \left.
\begin{array}{c}
 j_2 \\ \vdots \\ j_k \\ \vdots \\ j_h
\end{array}
\right]^*  \big)
\end{align*}

\begin{align*}
\phantom{ \quad \ \mathcal{K} \big(  \left[
\begin{array}{c}
i_1
\end{array}
\right| \left.
\begin{array}{c}
j_1
\end{array}
\right]^*  \big) =}&
\\
=&
\ \rho_{i_1j_1}  \big( \mathcal{K} \big(  \left[
\begin{array}{c}
  i_2 \\ \vdots \\ i_h
\end{array}
\right| \left.
\begin{array}{c}
  j_2 \\ \vdots \\ j_h
\end{array}
\right]^*  \big) \big) -
 \mathcal{K} \big( \sum_{k=2}^h \ \delta_{i_kj_1} \ \left[
\begin{array}{c}
 i_2 \\ \vdots \\ i_1 \\ \vdots \\ i_h
\end{array}
\right| \left.
\begin{array}{c}
 j_2 \\ \vdots \\ j_k \\ \vdots \\ j_h
\end{array}
\right]^*  \big)
\\
=&
\  D^{\textit{l}}_{i_1j_1} \big( \left(
\begin{array}{c}
 i_2 \\ \vdots \\ i_h
\end{array}
\right| \left.
\begin{array}{c}
  j_2 \\ \vdots \\ j_h
\end{array}
\right)^* \big) + (i_1|j_1)\left(
\begin{array}{c}
 i_2 \\ \vdots \\ i_h
\end{array}
\right| \left.
\begin{array}{c}
  j_2 \\ \vdots \\ j_h
\end{array}
\right)^*
\\
&
 \qquad \qquad - \sum_{k=2}^h \ \delta_{i_kj_1} \ \left(
\begin{array}{c}
 i_2 \\ \vdots \\ i_1 \\ \vdots \\ i_h
\end{array}
\right| \left.
\begin{array}{c}
 j_2 \\ \vdots \\ j_k \\ \vdots \\ j_h
\end{array}
\right)^*
\\
=&
(i_1|j_1) \left(
\begin{array}{c}
 i_2 \\ \vdots \\ i_h
\end{array}
\right| \left.
\begin{array}{c}
  j_2 \\ \vdots \\ j_h
\end{array}
\right)^* =
\left(
\begin{array}{c}
i_1\\  i_2 \\ \vdots \\ i_h
\end{array}
\right| \left.
\begin{array}{c}
j_1\\  j_2 \\ \vdots \\ j_h
\end{array}
\right)^*.
\end{align*}
\end{proof}

Notice that  Theorem \ref{operator B} specializes to
\begin{equation}\label{B column}
\mathcal{B} \big( \left(
\begin{array}{c}
i_1\\  i_2 \\ \vdots \\ i_h
\end{array}
\right| \left.
\begin{array}{c}
j_1\\  j_2 \\ \vdots \\ j_h
\end{array}
\right) \big) = \left[
\begin{array}{c}
i_1\\  i_2 \\ \vdots \\ i_h
\end{array}
\right| \left.
\begin{array}{c}
j_1\\  j_2 \\ \vdots \\ j_h
\end{array}
\right]
\end{equation}
and,   Theorem \ref{operator B^*} specializes to
\begin{equation}\label{B^* column}
\mathcal{B^*}\big( \left(
\begin{array}{c}
i_1\\  i_2 \\ \vdots \\ i_h
\end{array}
\right| \left.
\begin{array}{c}
j_1\\  j_2 \\ \vdots \\ j_h
\end{array}
\right)^* \big) = \left[
\begin{array}{c}
i_1\\  i_2 \\ \vdots \\ i_h
\end{array}
\right| \left.
\begin{array}{c}
j_1\\  j_2 \\ \vdots \\ j_h
\end{array}
\right]^*.
\end{equation}

\begin{theorem}\label{BCKtheorem} We have:

\begin{enumerate}

\item  $\mathcal{B} = \mathcal{K}^{-1}$,

\item  $\mathcal{B^*} = \mathcal{K}^{-1}$,

\item  $\mathcal{B},  \ \mathcal{B^*}, \ \mathcal{K}$ are linear isomorphisms,

\item  $\mathcal{B} = \mathcal{B^*}$.

\end{enumerate}
\end{theorem}
\begin{proof} From Corollary \ref{generators} and Eqs. (\ref{B column}) and  (\ref{B^* column}),
it follows that the operators $\mathcal{B}$ and $\mathcal{B^*}$ are surjective. 
Since column bitableaux
span ${\mathbb C}[M_{n,n}]$, Propositions \ref{K column} and \ref{K *-column}  imply that $\mathcal{B}$ and $\mathcal{B^*}$ are injective and $\mathcal{B} = \mathcal{K}^{-1}$ and $\mathcal{B^*} = \mathcal{K}^{-1}$. 
Then $\mathcal{B} = \mathcal{B^*}$.
\end{proof} 

By combining Theorems \ref{operator B} and \ref{operator B^*} with Theorem \ref{BCKtheorem}, 
it follows
\begin{corollary}\label{K action} We have:
\begin{itemize}

\item [--] $\mathcal{K} : [S|T] \mapsto (S|T),$

\item [--] $\mathcal{K} : [S|T]^* \mapsto (S|T)^*.$
\end{itemize}

\end{corollary}\qed

The Koszul isomorphism
$\mathcal{K}$ well-behaves with respect to  \emph{right symmetrized  bitableaux}
and {\textit{right Young-Capelli bitableaux}}.

\begin{proposition}\label{image symm} We have:
$$
\mathcal{K} : [S| \fbox{$T$} ] \mapsto (S|\fbox{$T$}).
$$
\end{proposition}
\begin{proof}
We notice that
$$
[S|\fbox{$T$})] = \sum_{\overline{T}} \ [S|\overline{T}], \quad 
(S|\fbox{$T$})) = \sum_{\overline{T}} \ (S|\overline{T}),
$$
where the sum is extended over {\textit{all}} $\overline{T}$ column permuted of $T$ (hence, repeated entries in
a column give rise to multiplicities).
The proof of the first equality easily follows from the definition, by applying the commutator identities
in the superalgebra ${\mathbf{U}}(gl(m_0|m_1+n))$.
The second equality is the definition of the right symmetrized bitableau 
$(S|\fbox{$T$}))$, Eq. (\ref{symmetrized bitableau}).
\end{proof}

From Proposition \ref{bases thms}, Corollary \ref{K action} and Proposition \ref{image symm}, it follows
\begin{corollary} The sets of standard  Capelli bitableaux,
of costandard  Capelli *-bitableaux
and  of standard Young-Capelli bitableaux:
\begin{itemize}

\item [--] $\Big \{ [S|T]; \ sh(S) =sh(T) =  \lambda, \ \lambda_1 \leq n, \ S,  T \ standard \Big \}$,

\item [--] $\Big \{ [U|V]^*; \ sh(U) =sh(V) =  \mu, \ \widetilde{\mu_1} \leq n, \ U,  V \ costandard \Big \}$,

\item [--] $\Big\{ [S|\fbox{$T$}]; \  \ sh(S) = sh(T) = \lambda, \ \lambda_1 \leq n, \ S, T \ standard \Big\}$

\end{itemize}
are linear bases of $\mathbf{U}(gl(n))$.

\end{corollary}

Furthermore, we have
\begin{theorem}\label{equivariant}  The Koszul isomorphism $\mathcal{K}$ is \emph{equivariant} with respect 
to the adjoint representations $\big( Ad_{gl(n)},  \mathbf{U}(gl(n)) \big)$ and
$\big( ad_{gl(n)},  {\mathbb C}[M_{n,n}] \big)$.
\end{theorem}
\begin{proof} We recall that the action of $e_{hk} \in gl(n)$ on $\mathbf{U}(gl(n))$ 
through the adjoint representation $Ad_{gl(n)}$ is implemented by the 
derivation $T_{hk}$ such that  $T_{hk}\big(e_{st}\big) = \delta_{ks} e_{it} - \delta_{ht} e_{sj}$.
From the  definition of column Capelli bitableau and Proposition \ref{rappresentazione aggiunta-BR}, we infer
$$
T_{hk} \Big(
\left[
\begin{array}{c}
i_1\\  i_2 \\ \vdots \\ i_h
\end{array}
\right| \left.
\begin{array}{c}
j_1\\  j_2 \\ \vdots \\ j_h
\end{array}
\right]
\Big) = \sum_{p=1}^h \ \delta_{k, i_p}
\left[
\begin{array}{c}
i_1 \\ \vdots \\ h \\ \vdots \\ i_h
\end{array}
\right| \left.
\begin{array}{c}
j_1\\ \vdots \\ j_p \\ \vdots \\ j_h
\end{array}
\right]
-
\sum_{p=1}^h \ \delta_{j_p, h}
\left[
\begin{array}{c}
i_1 \\ \vdots \\ i_p \\ \vdots \\ i_h
\end{array}
\right| \left.
\begin{array}{c}
j_1\\ \vdots \\ k \\ \vdots \\ j_h
\end{array}
\right].
$$
We recall that the action of $e_{hk}$ on ${\mathbb C}[M_{n,n}]$ 
through the adjoint representation $ad_{gl(n)}$ is implemented by the 
derivation $D^{\textit{l}}_{hk} - D^{\textit{r}}_{kh}$. Then
\begin{align*}
\big(D^{\textit{l}}_{hk}-D^{\textit{r}}_{kh} \big) \Big(
\left(
\begin{array}{c}
i_1\\  i_2 \\ \vdots \\ i_h
\end{array}
\right| \left.
\begin{array}{c}
j_1\\  j_2 \\ \vdots \\ j_h
\end{array}
\right)
\Big) &= \sum_{p=1}^h \ \delta_{k, i_p}
\left(
\begin{array}{c}
i_1 \\ \vdots \\ h \\ \vdots \\ i_h
\end{array}
\right| \left.
\begin{array}{c}
j_1\\ \vdots \\ j_p \\ \vdots \\ j_h
\end{array}
\right)
\\
&-\sum_{p=1}^h \ \delta_{j_p, h}
\left(
\begin{array}{c}
i_1 \\ \vdots \\ i_p \\ \vdots \\ i_h
\end{array}
\right| \left.
\begin{array}{c}
j_1\\ \vdots \\ k \\ \vdots \\ j_h
\end{array}
\right).
\end{align*}
Since column Capelli bitableaux span $\mathbf{U}(gl(n))$ and 
column  bitableaux span ${\mathbb C}[M_{n,n}]$, the assertion
follows from Proposition \ref{K column}.
\end{proof}
Since
$$
\boldsymbol{\zeta}(n) = \mathbf{U}(gl(n))^{Ad_{gl(n)}},
$$
the preceding Theorem implies:
\begin{corollary}\label{invariant} We have
$$
\mathcal{K} \Big[ \boldsymbol{\zeta}(n) \Big] = {\mathbb C}[M_{n,n}]^{ad_{gl(n)}}.
$$
\end{corollary}

In the left representation $\big( \rho^{\textit{l}}, {\mathbb C}[M_{n,n}] \big)$ 
(i.e. $\rho^{\textit{l}} :e_{ij} \mapsto D^{\textit{l}}_{ij}$),
standard  Young-Capelli bitableaux
$[S|\fbox{$T$}]$,  $sh(S) = sh(T) = \lambda \vdash k$,
act on right symmetrized bitableaux
$(U|\fbox{$V$})$, $sh(U) = sh(V) = \mu \vdash h$, in a quite remarkable way.
\begin{proposition} \emph{\cite{Brini2-BR}}\label{action}
We have:

\begin{itemize}
\item [--]
If $h < k$, the action is zero.

\item [--]
If $h = k$ and $\lambda \neq \mu$, the action is zero.

\item [--]
If $h = k$ and $\lambda = \mu$, the action is nondegenerate triangular
(with respect to a suitable linear order on standard tableaux of the same shape).
\end{itemize}
\end{proposition}
For details and proof, see \cite{Bri-BR} Theorem $10.1$. Clearly, similar results hold 
for the right and the adjoint representations.

\section{Laplace expansions}

\subsection{Laplace expansions in ${\mathbb C}[M_{n,n}]$}\label{LaplExp}

Recall that
$$
(i_1 i_2 \cdots i_h|j_1 j_2 \cdots j_h)
=
(-1)^{h \choose 2}  \ det[ (i_s|j_t) ]_{s,t =1, 2, \ldots, h}
\in {\mathbb C}[M_{n,n}],
$$
and, therefore, the biproduct $(i_1 i_2 \cdots i_h|j_1 j_2 \cdots j_h) \in {\mathbb C}[M_{n,n}]$ expands
into column bitableaux as follows:
\begin{align*}
(i_1 i_2 \cdots i_h|j_1 j_2 \cdots j_h) &=
\sum_{\sigma \in \mathbf{S}_h} \ (-1)^{|\sigma|} \left(
\begin{array}{c}
i_{\sigma(1)}\\  i_{\sigma(2)} \\ \vdots \\ i_{\sigma(h)}
\end{array}
\right| \left.
\begin{array}{c}
j_1\\  j_2 \\ \vdots \\ j_h
\end{array}
\right)
\\
&=
\sum_{\sigma \in \mathbf{S}_h} \ (-1)^{|\sigma|}  \left(
\begin{array}{c}
i_1\\  i_2 \\ \vdots \\ i_h
\end{array}
\right| \left.
\begin{array}{c}
j_{\sigma(1)}\\  j_{\sigma(2)} \\ \vdots \\ j_{\sigma(h)}
\end{array}
\right).
\end{align*}
Notice that, in the passage from monomials to column bitableaux,
the sign $(-1)^{h \choose 2}$ disappears.

Recall that
$$
(i_1 i_2 \cdots i_h|j_1 j_2 \cdots j_h)^*
=
 per[ (i_s|j_t) ]_{s,t =1, 2, \ldots, h}
\in {\mathbb C}[M_{n,n}],
$$
and, therefore, the *-biproduct $(i_1 i_2 \cdots i_h|j_1 j_2 \cdots j_h)^* \in {\mathbb C}[M_{n,n}]$ expands
into column *-bitableaux as follows:
\begin{align*}
(i_1 i_2 \cdots i_h|j_1 j_2 \cdots j_h)^* &=
\sum_{\sigma \in \mathbf{S}_h} \  \left(
\begin{array}{c}
i_{\sigma(1)}\\  i_{\sigma(2)} \\ \vdots \\ i_{\sigma(h)}
\end{array}
\right| \left.
\begin{array}{c}
j_1\\  j_2 \\ \vdots \\ j_h
\end{array}
\right)^*
\\
&=
\sum_{\sigma \in \mathbf{S}_h} \   \left(
\begin{array}{c}
i_1\\  i_2 \\ \vdots \\ i_h
\end{array}
\right| \left.
\begin{array}{c}
j_{\sigma(1)}\\  j_{\sigma(2)} \\ \vdots \\ j_{\sigma(h)}
\end{array}
\right)^*.
\end{align*}

The preceding arguments extend to  bitableaux and to  *-bitableaux of any shape $\lambda, \ \lambda_1 \leq n.$
Given the Young tableaux  
$$
S = \left(
\begin{array}{llllllllllllll}
i_{p_1}  \ldots    \ldots     \ldots     i_{p_{\lambda_1}} \\
i_{q_1}   \ldots  \ldots               i_{q_{\lambda_2}} \\
 \ldots  \ldots   \\
i_{r_1} \ldots i_{r_{\lambda_m}}
\end{array}
\right),
\quad
T = \left(
\begin{array}{llllllllllllll}
j_{s_1}  \ldots    \ldots     \ldots     j_{s_{\lambda_1}} \\
j_{t_1}   \ldots  \ldots               j_{t_{\lambda_2}} \\
 \ldots  \ldots   \\
j_{v_1} \ldots j_{v_{\lambda_m}}
\end{array}
\right).
$$
From a simple sign computation, it follows 
\begin{proposition}\label{exp1}
$$
(S|T) =
\sum_{\sigma_1, \ldots, \sigma_m } \ (-1)^{\sum_{k=1}^m \ |\sigma_k|} \
\left(
\begin{array}{c}
i_{p_{\sigma_1(1)}}\\   . \\ i_{p_{\sigma_1(\lambda_1)}} \\
\vdots   \\
i_{r_{\sigma_m(1)}}\\   . \\ i_{r_{\sigma_m(\lambda_m)}}
\end{array}
\right| \left.
\begin{array}{c}
j_{s_1}\\   . \\ j_{s_{\lambda_1}}  \\
\vdots \\
j_{v_1}\\   . \\ j_{v_{\lambda_m}}
\end{array}
\right),
$$
$$
\phantom{(S|T)} =
\sum_{\sigma_1, \ldots, \sigma_m } \ (-1)^{\sum_{k=1}^m \ |\sigma_k|} \
\left(
\begin{array}{c}
i_{p_1}\\   . \\ i_{p_{\lambda_1}}  \\
\vdots \\
i_{r_1}\\   . \\ i_{r_{\lambda_m}}
\end{array}
\right| \left.
\begin{array}{c}
j_{s_{\sigma_1(1)}}\\   . \\ j_{s_{\sigma_1(\lambda_1)}} \\
\vdots \\
j_{v_{\sigma_m(1)}}\\   . \\ j_{v_{\sigma_m(\lambda_m)}}
\end{array}
\right),
$$
where the multiple sums range over all permutations
$\sigma_1 \in \mathbf{S}_{\lambda_1}, \ldots, \sigma_m \in \mathbf{S}_{\lambda_m}.$
\end{proposition}
Notice that, in the expansions with respect to column bitableax, only the signs of permutations will remain.

Similarly, we have
\begin{proposition}\label{exp2}
$$
(S|T)^* =
\sum_{\sigma_1, \ldots, \sigma_m }  \
\left(
\begin{array}{c}
i_{p_{\sigma_1(1)}}\\   . \\ i_{p_{\sigma_1(\lambda_1)}} \\
\vdots \\
i_{r_{\sigma_m(1)}}\\   . \\ i_{r_{\sigma_m(\lambda_m)}}
\end{array}
\right| \left.
\begin{array}{c}
j_{s_1}\\   . \\ j_{s_{\lambda_1}}  \\
\vdots \\
j_{v_1}\\   . \\ j_{v_{\lambda_m}}
\end{array}
\right)^*
$$
$$
\phantom{(S|T)} =
\sum_{\sigma_1, \ldots, \sigma_m } \
\left(
\begin{array}{c}
i_{p_1}\\   . \\ i_{p_{\lambda_1}}  \\
\vdots \\
i_{r_1}\\   . \\ i_{r_{\lambda_m}}
\end{array}
\right| \left.
\begin{array}{c}
j_{s_{\sigma_1(1)}}\\   . \\ j_{s_{\sigma_1(\lambda_1)}} \\
\vdots \\
j_{v_{\sigma_m(1)}}\\   . \\ j_{v_{\sigma_m(\lambda_m)}}
\end{array}
\right)^*.
$$
\end{proposition}

\subsection{Laplace expansions in $\mathbf{U}(gl(n))$}

Let $S$ and $T$ be the Young tableaux
$$
S = \left(
\begin{array}{llllllllllllll}
i_{p_1}  \ldots    \ldots     \ldots     i_{p_{\lambda_1}} \\
i_{q_1}   \ldots  \ldots               i_{q_{\lambda_2}} \\
 \ldots  \ldots   \\
i_{r_1} \ldots i_{r_{\lambda_m}}
\end{array}
\right),
\quad
T = \left(
\begin{array}{llllllllllllll}
j_{s_1}  \ldots    \ldots     \ldots     j_{s_{\lambda_1}} \\
j_{t_1}   \ldots  \ldots               j_{t_{\lambda_2}} \\
 \ldots  \ldots   \\
j_{v_1} \ldots j_{v_{\lambda_m}}
\end{array}
\right).
$$

Propositions \ref{exp1} and \ref{exp2} and Theorems \ref{operator B} and \ref{operator B^*}
imply to the following \emph{Laplace expansions} of Capelli bitableaux
into column Capelli bitableaux and of Capelli *-bitableaux
into column Capelli *-bitableaux.

\begin{corollary}\label{Exp1} We have
$$
[S|T] =
\sum_{\sigma_1, \ldots, \sigma_m } \ (-1)^{\sum_{k=1}^m \ |\sigma_k|} \
\left[
\begin{array}{c}
i_{p_{\sigma_1(1)}}\\   . \\ i_{p_{\sigma_1(\lambda_1)}} \\
\vdots   \\
i_{r_{\sigma_m(1)}}\\   . \\ i_{r_{\sigma_m(\lambda_m)}}
\end{array}
\right| \left.
\begin{array}{c}
j_{s_1}\\   . \\ j_{s_{\lambda_1}}  \\
\vdots \\
j_{v_1}\\   . \\ j_{v_{\lambda_m}}
\end{array}
\right]
$$
$$
\phantom{[S|T]} =
\sum_{\sigma_1, \ldots, \sigma_m } \ (-1)^{\sum_{k=1}^m \ |\sigma_k|} \
\left[
\begin{array}{c}
i_{p_1}\\   . \\ i_{p_{\lambda_1}}  \\
\vdots \\
i_{r_1}\\   . \\ i_{r_{\lambda_m}}
\end{array}
\right| \left.
\begin{array}{c}
j_{s_{\sigma_1(1)}}\\   . \\ j_{s_{\sigma_1(\lambda_1)}} \\
\vdots \\
j_{v_{\sigma_m(1)}}\\   . \\ j_{v_{\sigma_m(\lambda_m)}}
\end{array}
\right].
$$
\end{corollary}

\begin{corollary}\label{Exp2} We have
$$
[S|T]^* =
\sum_{\sigma_1, \ldots, \sigma_m }  \
\left[
\begin{array}{c}
i_{p_{\sigma_1(1)}}\\   . \\ i_{p_{\sigma_1(\lambda_1)}} \\
\vdots   \\
i_{r_{\sigma_m(1)}}\\   . \\ i_{r_{\sigma_m(\lambda_m)}}
\end{array}
\right| \left.
\begin{array}{c}
j_{s_1}\\   . \\ j_{s_{\lambda_1}}  \\
\vdots \\
j_{v_1}\\   . \\ j_{v_{\lambda_m}}
\end{array}
\right]^*
$$
$$
\phantom{[S|T]^*} =
\sum_{\sigma_1, \ldots, \sigma_m }  \
\left[
\begin{array}{c}
i_{p_1}\\   . \\ i_{p_{\lambda_1}}  \\
\vdots \\
i_{r_1}\\   . \\ i_{r_{\lambda_m}}
\end{array}
\right| \left.
\begin{array}{c}
j_{s_{\sigma_1(1)}}\\   . \\ j_{s_{\sigma_1(\lambda_1)}} \\
\vdots \\
j_{v_{\sigma_m(1)}}\\   . \\ j_{v_{\sigma_m(\lambda_m)}}
\end{array}
\right]^*.
$$
\end{corollary}

By combining the expansions of Corollaries \ref{Exp1} and \ref{Exp2} with the results of
Proposition \ref{column exp}, one gets  explicit expansions of Capelli bitableaux 
and of Capelli *-bitableaux as elements of $\mathbf{U}(gl(n))$.

\begin{example}\label{excapbit} The Capelli bitableau (of shape $\lambda = (2, 2)$) 
\begin{equation}\label{excapbit display}
\left[
\begin{array}{ccc}
1 & 2
\\
2 & 4
\end{array}
\right| \left.
\begin{array}{ccc}
2 & 3
\\
3 & 4
\end{array}
\right]
 \in \mathbf{U}(gl(4))
\end{equation}
equals 
$$
\left[
\begin{array}{ccc}
1 \\ 2
\\ 2 \\ 4
\end{array}
\right| \left.
\begin{array}{ccc}
2 \\ 3
\\ 3 \\ 4
\end{array}
\right]
-
\left[
\begin{array}{ccc}
1 \\ 2
\\ 2 \\ 4
\end{array}
\right| \left.
\begin{array}{ccc}
3 \\ 2
\\ 3 \\ 4
\end{array}
\right]
-
\left[
\begin{array}{ccc}
1 \\ 2
\\ 2 \\ 4
\end{array}
\right| \left.
\begin{array}{ccc}
2 \\ 3
\\ 4 \\ 3
\end{array}
\right]
+
\left[
\begin{array}{ccc}
1 \\ 2
\\ 2 \\ 4
\end{array}
\right| \left.
\begin{array}{ccc}
3 \\ 2
\\ 4 \\ 3
\end{array}
\right],
$$
where
\begin{align*}
& \left[
\begin{array}{ccc}
1 \\ 2
\\ 2 \\ 4
\end{array}
\right| \left.
\begin{array}{ccc}
2 \\ 3
\\ 3 \\ 4
\end{array}
\right] =
\ e_{12}e_{23}e_{23}e_{44} - 2e_{13}e_{23}e_{44},
\\
& \left[
\begin{array}{ccc}
1 \\ 2
\\ 2 \\ 4
\end{array}
\right| \left.
\begin{array}{ccc}
3 \\ 2
\\ 3 \\ 4
\end{array}
\right] =
\ e_{13}e_{22}e_{23}e_{44} - e_{13}e_{23}e_{44},
\\
& \left[
\begin{array}{ccc}
1 \\ 2
\\ 2 \\ 4
\end{array}
\right| \left.
\begin{array}{ccc}
2 \\ 3
\\ 4 \\ 3
\end{array}
\right] =
\ e_{12}e_{23}e_{24}e_{43} - e_{12}e_{23}e_{23} - e_{13}e_{24}e_{43}
+ e_{13}e_{23} - e_{23}e_{14}e_{43} + e_{23}e_{13},
\\
& \left[
\begin{array}{ccc}
1 \\ 2
\\ 2 \\ 4
\end{array}
\right| \left.
\begin{array}{ccc}
3 \\ 2
\\ 4 \\ 3
\end{array}
\right] =
\  e_{13}e_{22}e_{24}e_{43} - e_{13}e_{22}e_{23} - e_{13}e_{24}e_{43}
+ e_{13}e_{23}.
\end{align*}

This example can be used to enlighten the difference between the PBW Theorem and  Theorem \ref{BCKtheorem}.

The PBW Theorem establishes an isomorphism $\phi$ from the \emph{graded algebra}
$$
Gr \left[ \mathbf{U}(gl(n)) \right] = \bigoplus_{h \in \mathbb{Z}^+} \ 
\frac {\mathbf{U}^{(h)}(gl(n))} {\mathbf{U}^{(h-1)}(gl(n))}
$$
associated to the \emph{filtered} algebra $\mathbf{U}(gl(n))$ to the algebra
$\mathbf{Sym}(gl(n)) \cong {\mathbb C}[M_{n,n}] $. Clearly, the isomorphism $\phi$ maps the
projection to $Gr \left[ \mathbf{U}(gl(n)) \right]$ of the
Capelli bitableau (\ref{excapbit display}) 
- \emph{as an element of the quotient space}
$\frac {\mathbf{U}^{(4)}(gl(4))} {\mathbf{U}^{(3)}(gl(4))}$ -
to  the product determinants
\begin{equation}\label{proddet}
\mathbf{det}
\left(
\begin{array}{ccc}
(1|2) & (1|3)
\\
(2|2) & (2|3)
\end{array}
\right)
\times
\mathbf{det}
\left(
\begin{array}{ccc}
(2|3) & (2|4)
\\
(4|3) & (4|4)
\end{array}
\right)  \in {\mathbb C}[M_{4,4}].
\end{equation}

The Koszul isomorphism $\mathcal{K}$ (injectively) maps the  Capelli bitableau (\ref{excapbit display})
- \emph{as an element of $\mathbf{U}(gl(4))$} - to the product of determinants (\ref{proddet}).
Similarly,  the isomorphism $\mathcal{K}$ maps the Capelli *-bitableau
$$
\left[
\begin{array}{ccc}
1 & 2
\\
2 & 4
\end{array}
\right| \left.
\begin{array}{ccc}
2 & 3
\\
3 & 4
\end{array}
\right]^* \in \mathbf{U}(gl(4))
$$
to the product permanents
$$
\mathbf{per}
\left(
\begin{array}{ccc}
(1|2) & (1|3)
\\
(2|2) & (2|3)
\end{array}
\right)
\times
\mathbf{per}
\left(
\begin{array}{ccc}
(2|3) & (2|4)
\\
(4|3) & (4|4)
\end{array}
\right)  \in {\mathbb C}[M_{4,4}].
$$
\end{example}\qed

In the following, we will discuss some implications of Corollary \ref{invariant}.

\begin{proposition} \emph{(Koszul \cite{Koszul-BR})} \label{CapDet} Consider the row Capelli bitableau
$$
[n \cdots 2 1 |12 \cdots n] \in \mathbf{U}(gl(n)).
$$
We have:
\begin{enumerate}

\item

\begin{align*}
[n \cdots 2 1 |12 \cdots n] =
\mathbf{cdet}
\left(
\begin{array}{cccc}
 e_{1 1}+(n-1) & e_{1 2} & \ldots  & e_{1 n} \\
 e_{2 1} & e_{2 2}+(n-2) & \ldots  & e_{2 n}\\
 \vdots  &    \vdots                            & \vdots &  \\
e_{n 1} & e_{n 2} & \ldots & e_{n n}\\
 \end{array}
 \right),
\end{align*}
the \emph{Capelli column determinant}\footnote{The symbol $\mathbf{cdet}$
denotes the column determinat of a matrix $A = [a_{ij}]$ with noncommutative entries:
$\mathbf{cdet} (A) = \sum_{\sigma} \ (-1)^{|\sigma|}  \ a_{\sigma(1), 1}a_{\sigma(2), 2} \cdots a_{\sigma(n), n}.$} in $\mathbf{U}(gl(n))$.

\item

\begin{align*}
\mathcal{K} \left( [n \cdots 2 1 |12 \cdots n] \right) =&
 \
 \mathbf{det}
 \left(
 \begin{array}{ccc}
 (1|1) & \ldots & (1|n) \\
 \vdots  &       & \vdots \\
 (n|1) & \ldots & (n|n) \\
 \end{array}
 \right) \in {\mathbb C}[M_{n,n}].
 \end{align*}

\end{enumerate}

\end{proposition}

\begin{proof}
We have

\begin{align*}
[n \cdots 2 1 &|12 \cdots n] =
\sum_{\sigma \in \mathbf{S}_n} \ (-1)^{|\sigma|}
\left[
\begin{array}{c}
\sigma(n) \\ \sigma(n-1) \\ \vdots \\ \sigma(1)
\end{array}
\right| \left.
\begin{array}{c}
1 \\ 2 \\ \vdots \\ n
\end{array}
\right]
\\
=&
\sum_{\sigma \in \mathbf{S}_n} \ (-1)^{|\sigma|} \times
\\
& \ 	\Big( (-1)^{n-1} \ e_{\sigma(n) 1}
\left[
\begin{array}{c}
\sigma(n-1) \\ \sigma(n-2) \\ \vdots \\ \sigma(1)
\end{array}
\right| \left.
\begin{array}{c}
2 \\ 3 \\ \vdots \\ n
\end{array}
\right]
\\
& + (-1)^{n-2}
\sum_{k=2}^{n} \ \delta_{\sigma(n-k+1) 1}
\left[
\begin{array}{c}
\sigma(n-1) \\ \vdots \\ \sigma(n) \\ \vdots \\ \sigma(1)
\end{array}
\right| \left.
\begin{array}{c}
2 \\  \vdots \\ k \\  \vdots \\ n
\end{array}
\right] \Big)
\\
=&
\ (-1)^{n-1} \
\sum_{\sigma \in \mathbf{S}_n} \ (-1)^{|\sigma|}
\left(e_{\sigma(n) 1} + (n-1) \delta_{\sigma(n) 1} \right)
\left[
\begin{array}{c}
\sigma(n-1) \\ \sigma(n-2) \\ \vdots \\ \sigma(1)
\end{array}
\right| \left.
\begin{array}{c}
2 \\ 3 \\ \vdots \\ n
\end{array}
\right],
\end{align*}
from  Proposition \ref{column exp}.

By iterating the same argument,
\begin{align*}
&[n \cdots 2 1 |12 \cdots n] =
\\
=&
\ (-1)^{n \choose 2} \times
\\
&
\sum_{\sigma \in \mathbf{S}_n} \ (-1)^{|\sigma|}
\left(e_{\sigma(n) 1} + (n-1) \delta_{\sigma(n) 1} \right)
\left(e_{\sigma(n-1) 2} + (n-2) \delta_{\sigma(n-1) 2} \right)
\cdots
e_{\sigma(1) n}
\\
=&
\sum_{\tau \in \mathbf{S}_n} \ (-1)^{|\tau|}
\left(e_{\tau(1) 1} + (n-1) \delta_{\tau(1) 1} \right)
\left(e_{\tau(2) 2} + (n-2) \delta_{\tau(2) 2} \right)
\cdots
e_{\tau(n) n}
\\
=&
\ \mathbf{cdet}
\left(
\begin{array}{cccc}
 e_{1 1}+(n-1) & e_{1 2} & \ldots  & e_{1 n} \\
 e_{2 1} & e_{2 2}+(n-2) & \ldots  & e_{2 n}\\
 \vdots  &    \vdots                            & \vdots &  \\
e_{n 1} & e_{n 2} & \ldots & e_{n n}\\
 \end{array}
 \right) \in \mathbf{U}(gl(n)),
\end{align*}
the  Capelli column determinant in $\mathbf{U}(gl(n))$.

Then
$$
\mathcal{K}   \Big(
\mathbf{cdet}
\left(
\begin{array} {cccc}
e_{1 1}+(n-1) & e_{1 2}       & \ldots  & e_{1 n} \\
 e_{2 1}      & e_{2 2}+(n-2) & \ldots  & e_{2 n} \\
 \vdots       &    \vdots     & \vdots  &  \vdots \\
e_{n 1}       & e_{n 2}       & \ldots  & e_{n n} \\
\end{array}
\right) \Big)
= \ \mathcal{K} \left( [n \cdots 2 1 |12 \cdots n] \right),
$$
that equals
$$
(n \cdots 2 1 |12 \cdots n) 
= \
 \mathbf{det}
 \left(
 \begin{array}{ccc}
 (1|1) & \ldots & (1|n) \\
 \vdots&        & \vdots \\
 (n|1) & \ldots & (n|n) \\
 \end{array}
 \right)  \in {\mathbb C}[M_{n,n}],
$$
by Corollary \ref{K action}.
\end{proof}

In the enveloping algebra $\mathbf{U}(gl(n))$, given any integer $k = 1, 2, \ldots, n,$ consider the 
$k-$th \emph{Capelli element}:
\begin{equation}\label{k Capelli}
\mathbf{H}_k(n)  = \
 \sum_{1 \leq i_1 < \cdots < i_k \leq n} \ [ i_k \cdots i_2 i_1 | i_1 i_2 \cdots i_k ].
\end{equation}

By the same argument of Proposition \ref{CapDet},
$$
\mathbf{H}_k(n)  = \ \sum_{1 \leq i_1 < \cdots < i_k \leq n}  \mathbf{cdet}\left(
 \begin{array}{cccc}
 e_{{i_1},{i_1}}+(k-1) & e_{{i_1},{i_2}} & \ldots  & e_{{i_1},{i_k}} \\
 e_{{i_2},{i_1}} & e_{{i_2},{i_2}}+(k-2) & \ldots  & e_{{i_2},{i_k}}\\
 \vdots          &      \vdots                     &     \vdots & \\
 e_{{i_k},{i_1}} & e_{{i_k},{i_2}}       & \ldots  & e_{{i_k},{i_k}}\\
 \end{array}
 \right),
$$
and the operator $\mathcal{K}$ maps $\mathbf{H}_k(n)$
to the polynomial
\begin{align*}
\mathbf{h}_k(n) =& \ \sum_{1 \leq i_1 < \cdots < i_k \leq n} \  ( i_k \cdots i_2 i_1 | i_1 i_2 \cdots i_k )
\\
=& \ \sum_{1 \leq i_1 < \cdots < i_k \leq n}  \mathbf{det}\left(
 \begin{array}{ccc}
(i_1|i_1) & \ldots  & (i_1|i_k) \\
\vdots   &          & \vdots    \\
(i_k|i_1) & \ldots  & (i_k|i_k) \\
 \end{array}
\right) \in {\mathbb C}[M_{n,n}].
\end{align*}

Notice that the polynomials $\mathbf{h}_k(n)$'s  appear as coefficients (in ${\mathbb C}[M_{n,n}]$)
of the characteristic polynomial:
$$
P_{M_{n,n}}(t) = det \big( tI - M_{n,n} \big) =
t^n + 	\sum_{i=1}^n \ (-1)^i \ \mathbf{h}_i(n) \ t^{n-i}.
$$

Clearly,  $\mathbf{h}_k(n)$ is $ad_{gl(n)}-$invariant in
${\mathbb C}[M_{n,n}]$ and, therefore,   $\mathbf{H}_k(n)$ is a
\emph{central element} of the enveloping algebra $\mathbf{U}(gl(n))$.

In passing we recall Capelli's Theorem (\cite{Cap1-BR} and \cite{Cap2-BR}, see also \cite{Brini4-BR}):
\begin{proposition}
$$
\boldsymbol{\zeta}(n) = {\mathbb C} \big[ \mathbf{H}_1(n), \mathbf{H}_2(n), \ldots, \mathbf{H}_n(n) \big].
$$
Moreover, the $\mathbf{H}_k(n)$'s are algebraically independent.
\end{proposition}

In general, given a  partition $\lambda = (\lambda_1, \lambda_2, \ldots, \lambda_p)$,
 $\lambda_1 \leq n$, consider the sum of Capelli
bitableaux
$$
\mathbf{K}_{\lambda}(n) = \sum_S \ [S|S],
$$
where the sum is extended to all row-increasing tableaux $S$, $sh(S) = \lambda$
(the $\mathbf{K}_{\lambda}(n)$'s are called \emph{shaped} Capelli elements in \cite{BriniTeolis-BR}).
Notice that the elements 
$\mathbf{K}_{\lambda}(n)$ are \emph{radically different}
from the elements 
$\mathbf{H}_{\lambda}(n) = \mathbf{H}_{\lambda_1}(n) \cdots \mathbf{H}_{\lambda_p}(n)$.

From Corollary \ref{K action}, Eqs. (\ref{bitableau}) and (\ref{crossing sign}) and \emph{row} skew-symmetry of bitableaux, we infer
\begin{proposition} We have
$$
\mathcal{K} \big( \mathbf{K}_{\lambda}(n) \big) = (-1)^{|\lambda| \choose 2} \ 
\mathbf{h}_{\lambda_1}(n)\mathbf{h}_{\lambda_2}(n) 
\cdots \mathbf{h}_{\lambda_p}(n), \quad |\lambda| = \sum_i \ \lambda_i.
$$
\end{proposition}
Hence, the elements  $\mathbf{K}_{\lambda}(n)$ are central. By Corollary \ref{invariant},
the following statements are equivalent:

\begin{itemize}

\item [--] The $\mathbf{K}_{\lambda}(n)-$basis theorem for $\boldsymbol{\zeta}(n)$ \cite{BriniTeolis-BR}:
\begin{proposition}\label{K basis} The set
$$
\Big\{\mathbf{K}_{\lambda}(n); \ \lambda_1 \leq n \Big\}
$$
is a linear basis of $\boldsymbol{\zeta}(n)$. 
\end{proposition}
Notice that the elements $\mathbf{K}_{\lambda}(n)$ are \emph{radically different} 
from the \emph{quantum immanants} of \cite{Okounkov-BR}, \cite{Okounkov1-BR} and \cite{Brini5-BR}.
\item [--]
The  well-known theorem 
for the algebra of invariants ${\mathbb C}[M_{n,n}]^{ad_{gl(n)}}$:
\begin{proposition} \label{First Fund}
$$
{\mathbb C}[M_{n,n}]^{ad_{gl(n)}} = 
{\mathbb C} \big[ \mathbf{h}_1(n), \mathbf{h}_2(n), \ldots, \mathbf{h}_n(n) \big].
$$
Moreover, the $\mathbf{h}_k(n)$'s are algebraically independent.
\end{proposition}
Proposition \ref{First Fund} is usually stated in terms of the
algebra ${\mathbb C}[M_{n,n}]^{GL(n)} = {\mathbb C}[M_{n,n}]^{ad_{gl(n)}}$, where
${\mathbb C}[M_{n,n}]^{GL(n)}$ is the subalgebra of invariants with respect to the \emph{conjugation
action} of the general linear group ${GL(n)}$ on ${\mathbb C}[M_{n,n}]$
(see, e.g. \cite{KP-BR}, \cite{DEP-BR}, \cite{Procesi-BR}).
\end{itemize}

\end{document}